\documentclass[12pt,oneside]{amsart}

\usepackage{amssymb,enumerate,amscd,anysize,color}
\newtheorem{theorem}{Theorem}[section]
\newtheorem{lemma}[theorem]{Lemma}
\newtheorem{cor}[theorem]{Corollary}
\newtheorem{prop}[theorem]{Proposition}
\theoremstyle{definition}

\theoremstyle{remark}
\newtheorem{remark}{Remark}[section]

\def\R{{\mathbb R}}
\def\C{{\mathbb C}}

\def\Symtwor{\rm Sym(2,\R)}

\def\Gldr{GL(d,\R)}
\def\Spdr{Sp(d,\R)}

\def\spnr{{\mathfrak {sp}}(d,\R)}

\def\sotwo{{\mathfrak {so}}(2)}

\def\gldr{{\mathfrak {gl}}(d,\R)}

\def\gltwonr{{\mathfrak {gl}}(2d,\R)}

\def\Spnr{Sp(d,\R)}
\def\Sptwor{Sp(2,\R)}

\def\Symdr{\rm Sym(d,\R)}
\def\Symtwor{\rm Sym(2,\R)}

\def\gh{{\mathfrak h}}
\def\gk{{\mathfrak k}}
\def\gl{{\mathfrak l}}

\def\gq{{\mathfrak q}}

\def\sp{\mathop{\rm span}}

\def\trace{\mathop{\rm tr}}

\def\eps{\varepsilon}

\def\t{\!\;^t\!}   % transpose %

\def\cB{{\mathcal B}}

\def\cE{{\mathcal E}}

\def\cT{{\mathcal T}}

\begin{document}

\title[Reproducing subgroups]{Reproducing subgroups  of $Sp(2,\R)$.\\ Part I:  algebraic classification}


%    author one information
\author{G.~S.~Alberti}
\address{G.~S.~Alberti\\
Mathematical Institute\\
24-29 St Giles'\\
Oxford\\
OX1 3LB, England}
\email{Giovanni.Alberti@maths.ox.ac.uk }
%\thanks{}
%    author two information
\author{L.~Balletti}
\address{L. Balletti, CNR-PSC\\
Corso F. Perrone 24\\
16152  Genova, Italy}
%\email{balletti@dima.unige.it}
%    author two information
\author{F.~De~Mari}
\address{F. De Mari, DIMA\\
Via Dodecaneso, 35
\\16146 Genova, Italy}
\email{demari@dima.unige.it}
%    author two information
\author{E.~De Vito}
\address{E. De Vito, DIMA \and INFN - Sezione di Genova\\
Via Dodecaneso, 35
\\16146 Genova, Italy}
\email{devito@dima.unige.it}

%    \subjclass is required.

\date{November 8, 2012}

\dedicatory{}

%    Abstract is required.
\begin{abstract}We classify the connected Lie subgroups of  the symplectic group $\Sptwor$ whose elements 
are matrices in block lower triangular form. The classification is up to conjugation within $\Sptwor$. Their study is motivated by the need of a unified approach to continuous 2D signal analyses, as those provided by wavelets and shearlets.
\end{abstract}

\maketitle
%\tableofcontents
\section{Introduction} The continuous wavelet transform \cite{Dau92,fuhr05,mallat09,CoMe97} and its many variants, such as, for example, the shearlet transform \cite{dakustte09,dastte10,gukula06,kula09}, lie in the background of a growing body of techniques, that may be collectively referred to as {\it signal analysis}, whose common feature is perhaps the decomposition of functions, primarily in $L^2(\R^d)$, by means of superpositions of projections along selected ``directions''.  Symmetry and finite dimensional geometry often play a prominent r\^ole  in the way in which these directions are generated or selected, and hence, with this notion of signal analysis, topological transformation groups and their representations provide a natural setup. In particular, the restriction of the metaplectic representation of $\Spdr$ to its Lie subgroups produces a wealth of useful reproducing formulae \cite{codenota06b,codenota06a}, all based on linear geometric actions either in the time or in the frequency domain, and is thus 
one of the most natural environments both for a unified approach and for the search of new 
strategies. In fact, the deep connections of the metaplectic representation with harmonic analysis in phase space is thoroughly investigated \cite{fol89,gro01}, and one of the keys to its understanding is the Wigner transform.

The central importance of the symplectic group  has motivated both a
general theory of ``mock'' metaplectic representations (and the
abstract harmonic analysis thereof  \cite{DeDe11}), and a more
applications-oriented approach, where the main focus  is the actual
study of these formulae in connection with the classical themes of
signal analysis \cite{ober10}. In this work, that consists of two
parts, we introduce the class $\cE$ of Lie subgroups of $\Spdr$ that
we believe is the ``right'' class  for signal analysis and we
illustrate its relevance in $2D$-analysis by exhibiting the full list
of reproducing formulae that it yields, up to the appropriate notion
of equivalence. In some sense, therefore, we obtain a complete
picture, at least as far as continuous ``geometric''  transforms are
concerned, of   reasonable $2D$ signal analyses. 
In the first part (this paper) we classify the groups,  modulo conjugation within $Sp(2,\R)$. In part~II we address the analytic issues: by appealing to the theory developed in \cite{DeDe11} we are able to show exactly which groups are reproducing and which are not. The  full description of the associated admissible vectors is also achieved.

 Some other interesting examples of signal analysis
  associated with the metaplectic representation in higher dimensions can
  be found in a recent paper \cite{czki12}. 

We  say that a Lie subgroup $G$ of $\Spdr$ is a {\it reproducing group} if there exists a function $\eta\in L^2(\R^d)$, to be called an {\it admissible vector},  such that  the  {\it reproducing formula}
\begin{equation*}
f=\int_G\langle f,\mu_g\eta\rangle\mu_g\eta\,dg,
\label{RF}
\end{equation*}
holds (weakly) for every $f\in L^2(\R^d)$, where $dg$ is a left Haar measure of $G$ and $\mu$ is the metaplectic representation restricted to $G$. For simplicity, we actually restrict ourselves to connected subgroups.  As pointed out in previous work \cite{codenota06b,codenota06a,DeDe11}, many known continuous formulae (notably those associated to wavelets, shearlets and some of their variants) arise in this way, or are at least   equivalent to them via natural intertwining operators such as the Fourier transform, perhaps combined with geometric (affine) transformations of phase space.  But much more is true. All the reproducing groups that we are aware of, share a structural feature:  they are block triangular\footnote{By conjugating with a suitable permutation one can either adopt the lower or upper triangular shape, as desired.}  semidirect products of a particular type. Written as $d\times d$ blocks, their elements have the form
\begin{equation*}
g(\sigma,h)=
\begin{bmatrix}h&0\\\sigma h&\,^th^{-1}\end{bmatrix}
\label{lowertriangular}
\end{equation*}
where $\sigma$ ranges in a nontrivial vector space $\Sigma$ of symmetric $d\times d$ matrices (the {\it vector} components) and  $h$ ranges, independently of $\sigma$,  in a nontrivial  connected Lie subgroup $H$ of $\Gldr$ (the {\it homogeneous}  component),  that acts on $\Sigma$ via
$$
h^{\dagger}[\sigma]=\,^th^{-1}\sigma h^{-1}.
$$
From the point of view of analysis, one should think of $\Sigma$ as
encoding translations and $H$ as the group of geometric
``deformations'' such as, for example, shearings or possibly
anisotropic dilations, or combinations of both.  Thus, a group in the
class $\cE$ is, by definition, a connected semidirect product
$G=\Sigma\rtimes H$.  All these groups lie inside the standard maximal
parabolic subgroup $Q$ of $\Spdr$ described in \eqref{maxPAR}, but, in
general, they are not parabolic, nor do they fill up the class of
connected Lie subgroups of $Q$, as we show below. As one realizes
after the classification, $\cE$ is a nontrivial class, perhaps
not much smaller than the family including all reproducing subgroups of
$\Spdr$: we actually conjecture that if $G$ is a connected reproducing
subgroup of $\Spdr$, then, modulo extensions by compact factors, $G$
is conjugate within $\Spdr$ to a Lie subgroup of $Q$. Furthermore, at
the present stage, the theory initiated in~\cite{DeDe11} treats only
semidirect products and is thus not applicable to the full class of
Lie subgroups of $Q$, which is described in
Proposition~\ref{triplette}. We therefore content ourselves with those
in the class $\cE$, which is actually rich enough to contain both
significant known cases and several new examples, at times surprising.

In the two papers, of which this is the first, we accomplish one of the main objectives of our research project,  namely the classification, for $d=2$,  of all the reproducing groups in $\cE$, together with the relevant analytic information. The classification we are after, of course, must be done modulo some reasonable and pertinent notion of equivalence. This is a rather delicate issue, as we now illustrate, and is one of the central points of our work.
The most natural notion of equivalence is algebraic. In Lie theoretic terms, it is just conjugation modulo $MA$, where $MAN$ is the Langlands decomposition of $Q$. The matrices in $MA$ are the block diagonal elements in $Q$ and conjugation by them preserves the class $\cE$. As explained in Proposition~\ref{conjugation}, every $y\in MA$ sends any $G\in\cE$ into $yGy^{-1}\in\cE$ and actually maps  vector components into vector components (i.e. $\Sigma$ to $\Sigma'$, because $MA$ normalizes $N$) and homogeneous components into homogeneous components (i.e. $H$ to $H'$, because $MA$ normalizes itself). No other symplectic matrix has this property on all of $\cE$. Furthermore, this equivalence yields the equivalence of the restrictions of the metaplectic representation,  groups in the same equivalence class are either all  reproducing or none of them is,  and the sets of admissible vectors in a reproducing class are in one-to-one correspondence via the unitary equivalences induced by $\mu(y)$.

Another natural equivalence is conjugation by {\it any}  element in $\Spdr$. It is very important because, although not adapted to $\cE$, any conjugation induces equivalence of the restrictions of the metaplectic representation, and  transfers the reproducing property, with  admissible vectors that correspond to eachother via  natural unitary equivalence. In Section~\ref{full} we analyze in full detail this general conjugation problem and we finally prove the  classification, which is our main result and is stated in the next section.

\section{Setting and main results}
We first introduce some notation.  The letters $g,h$ are to be
regarded as invertible matrices and $\sigma$ as a symmetric matrix.
We write
\begin{align*}
g^{\sharp}&=\,^tg^{-1}\\ 
i_g(h)&=ghg^{-1}\\
g^{\dagger}[\sigma]&
=\,^tg^{-1}\sigma g^{-1}.   
\end{align*}
If $G$ is Lie group, the connected component of the identity will be denoted $G^0$.

In this paper, we focus on a class $\cE$ of subgroups of the symplectic group 
$$
\Spdr=\{g\in GL(2d,\R):\,^tgJg=J\},
$$
where
$J=\left[\begin{smallmatrix} O &I\\-I&0\end{smallmatrix}\right]$ and
$I$ is the identity matrix, whose size will be clear from the context.
We look at block lower triangular matrices  of the form
\begin{equation*}
g(\sigma,h):=\begin{bmatrix}
    h & 0 \\
   \sigma\,h & h^\sharp 
  \end{bmatrix}\in \Spdr,
\label{elem}
\end{equation*}
where $\sigma\in \Symdr$,  the vector space of $d\times d$ real
symmetric matrices, and $h\in GL(d,\R)$.  We say that the group $G$ is
in the class $\cE$  if there exist a non-zero vector subspace $\Sigma$
of $\Symdr$ and a connected nontrivial Lie subgroup $H$ of $GL(d,\R)$ such that
\[
G=\{
  g(\sigma,h)  \,:\, \sigma\in\Sigma, h\in H\}.
\]
By construction,  $g(\Sigma,I)$ is an abelian normal Lie subgroup of
$G$ and is isomorphic to  $\Sigma$, $g(0,H)$ is a Lie subgroup of $G$ and is
isomorphic to $H$, and  $G$ is the semi-direct product of
$g(\Sigma,I)$ and $g(0,H)$.  Under this canonical identification, we write
$G=\Sigma\rtimes H$.  In Appendix~\ref{sec:parabolic} we further comment on the r\^ole
of $\cE$ in connection with the algebraic structure  of $\Spdr$.

We now introduce some specific notation needed for the  classification
of  the groups in the class $\cE$ with $d=2$.  We shall use the following
symmetric matrices
\begin{equation}
    \sigma_1=\begin{bmatrix}
               1 & 0 \\
               0 & 1 \\
             \end{bmatrix},\quad  \sigma_2=\begin{bmatrix}
               1 & 0 \\
               0 & -1 \\
             \end{bmatrix},\quad  \sigma_3=\begin{bmatrix}
               1 & 0 \\
               0 & 0 \\
             \end{bmatrix}
             ,\quad \sigma_4=\begin{bmatrix}
               0 & 0 \\
               0 & 1 \\
             \end{bmatrix}
             ,\quad \sigma_5=\begin{bmatrix}
               0 & 1 \\
               1 & 0 \\
             \end{bmatrix}.
             \label{5sigma}
\end{equation}
For $i=1,2,3$ we write $\Sigma_i=\sp\{\sigma_i\}$ for the corresponding
one dimensional subspace. Each of them defines a two
dimensional subspace,  the orthogonal complement with respect to
the scalar product induced by the trace on $\Symtwor$. Explicitly:
\begin{align*}
&\Sigma_1^\perp=\bigl\{\begin{bmatrix}
u &v \\
v&-u \\
\end{bmatrix}:u,v\in\R\bigr\}=\sp\{\sigma_2,\sigma_5\},\\
&\Sigma_2^\perp=\bigl\{\begin{bmatrix}
u &v \\
v&u \\
\end{bmatrix}:u,v\in\R\bigr\}=\sp\{\sigma_1,\sigma_5\},\\
&\Sigma_3^\perp=\bigl\{
\begin{bmatrix}
                           0 & v\\
                            v & u \\
                          \end{bmatrix}:
                          u,v\in\R\bigr\}=\sp\{\sigma_4,\sigma_5\}.
                         \end{align*}
For $t\in\R$ we set
\begin{equation*}
     R_t=\exp(tJ)=\begin{bmatrix}
               \cos t & \sin t \\
               -\sin t & \cos t \\
             \end{bmatrix},\quad  A_t=\exp(t\sigma_5)=\begin{bmatrix}
               \cosh t & \sinh t \\
               \sinh t & \cosh t \\
             \end{bmatrix}.
\end{equation*}
The notation relative to the connected Lie subgroups of $GL(2,\R)$ is as follows:
\begin{itemize}\addtolength{\itemsep}{0.1\baselineskip}
\item[] $SO(2)=\{R_\theta:\theta\in [0,2\pi )\}$
\item[] $SO^0(1,1)=\{A_t:t\in\R\}$
  \item[] $T=\bigl\{\left[\begin{smallmatrix}
c&0 \\
b&a \\
\end{smallmatrix}\right]
:a,b,c\in\R,\;ac\not=0\bigr\}$
  \item[] $H^0(\sigma_1)=SO(2)\times\R_+$
  \item[] $H_\infty(\sigma_1)=SO(2) $
  \item[] $H_\alpha(\sigma_1)=\{e^tR_{\alpha t}:t\in\R\},\qquad\alpha\in\R$
  \item[] $H^0(\sigma_2)=SO^0(1,1)\times\R_+$
  \item[] $H_\infty(\sigma_2)=SO^0(1,1)$
  \item[] $H_\alpha(\sigma_2)=\{e^t A_{\alpha t}:t\in\R\},\qquad\alpha\in\R$
  \item[] $H^0(\sigma_3)=T^0$
  \item[] $H_0(\sigma_3)=
\bigl\{\left[\begin{smallmatrix}
1 & 0 \\
t & 1 \\
\end{smallmatrix}\right]
:t\in\R\bigr\}$
  \item[] $H_1(\sigma_3)=\bigl\{e^t\left[\begin{smallmatrix}
                                         1 & 0 \\
                                         t & 1 \\
                                       \end{smallmatrix}\right]
:t\in\R\bigr\}$
  \item[] $H_\infty(\sigma_3)=\bigl\{e^t\left[\begin{smallmatrix}
                                         1 & 0 \\
                                         0 & 1 \\
                                       \end{smallmatrix}\right]
:t\in\R\bigr\}$
  \item[] $H_{\gamma,0}(\sigma_3)=\bigl\{\left[\begin{smallmatrix}
                                         e^{\gamma t} & 0 \\
                                         0 & e^{(\gamma+1)t} \\
                                       \end{smallmatrix}\right]
:t\in\R\bigr\},\qquad\gamma\in\R$
\item[] $K_0(\sigma_3)=
\bigl\{\left[\begin{smallmatrix}
e^t & 0 \\
0 & e^s \\
\end{smallmatrix}\right]
:s,t\in\R\bigr\}$
\item[] $K_\infty(\sigma_3)=\bigl\{\left[\begin{smallmatrix}
                                         e^t & 0 \\
                                         s & e^t \\
                                       \end{smallmatrix}\right]
:s,t\in\R\bigr\}$
\item[] $L_{\gamma}(\sigma_3)=\bigl\{\left[\begin{smallmatrix}
                                         e^{\gamma t} & 0 \\
                                         s & e^{(\gamma+1)t} \\
                                       \end{smallmatrix}\right]
:s,t\in\R\bigr\},\qquad\gamma\in\R.$
\end{itemize}

The following theorem gives, up to a conjugation within
$\Sptwor$, a complete list of the groups in $\cE$ when
 $1\leq\dim\Sigma\leq2$.  As shown by Theorem~1 of~\cite{DeDe11},  a group $G\in\cE$ with a  normal subgroup $\Sigma$ of dimension 3 is never reproducing, hence we avoid to
classify such groups.

\begin{theorem}\label{main}
 Let $d=2$ and take $G=\Sigma\rtimes H\in\cE$ with $1\leq\dim\Sigma\leq2$.
Then there exists $g\in Sp(2,\R)$ such that $gGg^{-1}$ is one the
following groups, each of which is in $\cE$ and none of which is conjugate to any other.\\
Two-dimensional groups:
\begin{enumerate}[(2.1)]
\item $\Sigma_1\rtimes H_\alpha(\sigma_1)$, $\alpha\in[0,+\infty]$
\item $\Sigma_2\rtimes H_\alpha(\sigma_2)$, $\alpha\in[0,+\infty]$
\item $\Sigma_3\rtimes H_0(\sigma_3)$
\item  $\Sigma_3\rtimes H_1(\sigma_3)$
\item $\Sigma_3\rtimes H_{\alpha,0}(\sigma_3)$, $\alpha\in[-1,0]$
\end{enumerate}
Three-dimensional groups:
\begin{enumerate}[(3.1)]
\item $\Sigma_1\rtimes H^0(\sigma_1)$
\item $\Sigma_2\rtimes H^0(\sigma_2)$
\item $\Sigma_3\rtimes K_0(\sigma_3)$
\item $\Sigma_3\rtimes K_{\infty}(\sigma_3)$
\item $\Sigma_3\rtimes L_{\gamma}(\sigma_3)$, $\gamma\in\R$
\item $\Sigma_1^\perp\rtimes H_\alpha(\sigma_1)$, $\alpha\in[0,+\infty]$
\item $\Sigma_2^\perp\rtimes H_\alpha(\sigma_2)$, $\alpha\in[0,+\infty]$
\item $\Sigma_3^\perp\rtimes \, ^tH_0(\sigma_3)$
\item $\Sigma_3^\perp\rtimes \,^tH_1(\sigma_3)$
\end{enumerate}
Four dimensional groups:
\begin{enumerate}[(4.1)]
\item $ \Sigma_3\rtimes H^0(\sigma_3)$
\item $\Sigma_1^\perp\rtimes H^0(\sigma_1)$
\item $\Sigma_2^\perp\rtimes H^0(\sigma_2)$
\item $\Sigma_3^\perp\rtimes\,^tL_{\gamma}(\sigma_3)$, $\gamma\in [-1,0]$
\end{enumerate}
Five dimensional groups:
\begin{enumerate}[(5.1)]
\item $\Sigma_3^\perp\rtimes \,^tH^0(\sigma_3)$.
\end{enumerate}
\end{theorem}

The proof of the above theorem is the content of the remaining part of
the paper.   We start with  some preliminary results, which are stated for
any size $d$.  

\section{Preliminary results}
First of all, we recall some basic algebraic properties
  of the symplectic group. For these and other standard Lie theoretic constructs and notions see for example 
  {\cite{knapp2002}}. The Lie algebra of $\Spdr$ is 
$$
\spnr=\left\{g\in\gltwonr:\;^t\!gJ+Jg=0\right\},
$$
and its elements are  of the form
\begin{equation*}
X=\begin{bmatrix}A&B\\C&-\t\,A\end{bmatrix},
\label{hamilton}
\end{equation*}
where $A$ is an arbitrary $d\times d$ matrix and $B,C\in\Symdr$. The
standard  maximal   parabolic subgroup $Q$ of the symplectic group
that we are interested in is the Lie group  
\begin{equation}
Q=\left\{\begin{bmatrix} h&0\\ \sigma h&h^\sharp\end{bmatrix}: h\in GL(d,\R), \;\sigma\in\Symdr  \right\},
\label{maxPAR}
\end{equation}
whose Lie algebra is:
\begin{equation*}
\gq=\left\{\begin{bmatrix} A&0\\ \sigma&-\t\,A\end{bmatrix}: A\in\gldr, \;\sigma\in\Symdr  \right\}.
\label{maxpar}
\end{equation*}
The Langlands decomposition $Q=MAN$ is easily checked to be
\begin{align*}
M
&=\left\{\begin{bmatrix} h&0\\0&\t\,h^{-1}\end{bmatrix}: \det h=\pm1  \right\}\\
A
&=\left\{\begin{bmatrix} \lambda I&0\\ 0&\lambda^{-1} I\end{bmatrix}: \lambda>0  \right\}
\\
N
&=\left\{\begin{bmatrix} I&0\\ \sigma&I\end{bmatrix}: \sigma\in\Symdr  \right\}.
\end{align*}
We  call  $MA\simeq GL(d,\R)$   the  homogeneous component and   $N\simeq\Symdr$  the vector component. As is well-known, $MA$ normalizes $N$, so that $Q$ is the semidirect product of $MA$ and the abelian normal factor $N$, namely
\begin{equation*}
Q=\Symdr\rtimes GL(d,\R),
\qquad
\gq=\Symdr\rtimes\gldr.
\label{maximals}
\end{equation*}
To see this explicitly, notice that each element of $Q$ is the product
\begin{equation*}
g(\sigma,h)= \begin{bmatrix} I&0\\  \sigma&I\end{bmatrix}
\begin{bmatrix}h&0\\ 0&h^\sharp\end{bmatrix}
=\begin{bmatrix} h&0\\ \sigma h&h^\sharp\end{bmatrix},
\end{equation*}
where $\sigma\in\Symdr$ and $h\in\Gldr$  and each such product is automatically symplectic.
The above factorization is formally
\begin{equation*}
g(\sigma,h)=g(\sigma,I)g(0,h).
\label{factor}
\end{equation*}
Now, the product of two matrices in $Q$ is
$$
g(\sigma,h)g(\sigma',h')
=\begin{bmatrix} hh'& 0\\ (\sigma +h^\dagger[\sigma'])hh'&(hh')^{\sharp}\end{bmatrix}
=g(\sigma +h^\dagger[\sigma'],hh'),
$$
where
\begin{equation}
h^\dagger[\sigma]=\t\,h^{-1}\sigma h^{-1}.
\label{cogr} 
\end{equation}
Thus, the group law is given by
\begin{equation}
g(\sigma,h)g(\sigma',h')=g(\sigma +h^\dagger[\sigma'],hh'),
\label{semidirect}
\end{equation}
 the identity is $g(0,I)$ and inverses are given by
\begin{equation*}
g(\sigma,h)^{-1}=g(-\,\!^th\sigma h,h^{-1})=g(-(h^{-1})^\dagger[\sigma],h^{-1}).
\label{inv}
\end{equation*}
Notice that
$$
\,^\dagger:\Gldr\times\Symdr\to\Symdr,
\qquad
\,^\dagger(h,\sigma)=h^\dagger[\sigma]
$$
is actually a group action and $\sigma\mapsto h^\dagger[\sigma]$ is a group automorphism of $N$.

Clearly, a group $G=\Sigma\rtimes H$ in the class $\cE$ is a connected Lie
  subgroup of $Q$ and the fact that $H$ normalizes $\Sigma$ is
  equivalent to the condition that $H$ leaves $\Sigma$ invariant under
  the action \eqref{cogr}. Conversely,  if $\Sigma$ is a non-zero subspace of
  $\Symdr$ and $H$ is a nontrivial Lie subgroup of $\Gldr$ such that
  $h^\dag[\Sigma]=\Sigma$ for all $h\in H$, then 
\[
G=\{ g(\sigma,h)\,:\,\sigma\in\Sigma,h\in H\}
\]
is a a connected Lie subgroup of $Q$ and is in the class
$\cE$. However, there exist connected Lie subgroups of $Q$, which are
not in the class $\cE$ (see some examples in
Appendix~\ref{sec:parabolic}).

We end this section with a few observations that  reduce our
classification problem.  As already noticed, by Theorem~1 of
\cite{DeDe11}, if $\Sigma\rtimes H\in\cE$ is a reproducing
group, then $1\leq\dim{\Sigma}\leq d$.  Furthermore, 
given a space $\Sigma$, there
exists a maximal closed (hence Lie) subgroup of $GL(d,\R)$ leaving
invariant $\Sigma$, namely
 \begin{equation}
H(\Sigma)=\left\{h\in GL(d,\R):h^\dagger[\sigma]\in\Sigma,\text{ for all }\sigma\in\Sigma\right\}.
\label{largest}
\end{equation}
Moreover,  there exists a ``duality relation'' induced by the orthogonality within $\Symdr$ relative to the usual inner product
$\langle\sigma,\tau\rangle=\trace(\sigma\tau)$. 
For any subset $\Sigma$ of ${\rm Sym}(d)$ we write
$$
\Sigma^\perp=\{\tau\in{\rm Sym}(d):\langle\sigma,\tau\rangle=0\text{ for all }\sigma\in\Sigma\}.
$$
As seen below in Proposition~\ref{perp}, the natural companion notion for the homogeneous factor $H$ is transposition. Hence
for any subgroup $H$ of $\Gldr$, we write
$$
\,^t\!H=\{\!\,^th:h\in H\}.
$$
\begin{prop}\label{perp} The following are equivalent:
\begin{itemize}
\item[(i)] $\Sigma\rtimes H\in\cE$;
\item[(ii)]  $\Sigma^\perp\rtimes\!\,^t\!H\in\cE$.
\end{itemize}
\end{prop}
\begin{proof}
If $\sigma\in\Sigma$, $\tau\in\Sigma^\perp$ and $h\in H$, then
$$
\langle (\!\,^th)^\dagger[\tau],\sigma\rangle
=\trace(h^{-1}\tau\!\,^th^{-1}\sigma)
=\trace(\tau\!\,^th^{-1}\sigma h^{-1})
=\langle \tau,h^\dagger[\sigma]\rangle.
$$
Therefore $h^\dagger[\Sigma]=\Sigma$ if and only if $ (\!\,^th)^\dagger[\Sigma^\perp]=\Sigma^\perp$.
\end{proof}

The next proposition shows that conjugation via $g(0,h)\in MA$ maps
$\cE$ into itself and, more precisely, that it preserves both  the
homogeneous and the normal factors. It also records that it  preserves
the subclass of reproducing groups.

\begin{prop}\label{conjugation} Take $\Sigma\rtimes H\in\cE$ and $h\in\Gldr$. Then $i_{g(0,h)}(\Sigma\rtimes H)\in\cE$. More precisely, if $\Sigma'\rtimes H'\in\cE$, then the following are equivalent:
\begin{itemize}
\item[(i)] $i_{g(0,h)}(\Sigma\rtimes H)=\Sigma'\rtimes H'$
\vskip0.1truecm
\item[(ii)]  $h^{\dagger}[\Sigma]=\Sigma'$ and  $i_h(H)=H'$.
\vskip0.2truecm
\item[(iii)] $i_{g(0,h^\sharp)}(\Sigma^\perp\rtimes\!\,^t\!H)=\Sigma'^\perp\rtimes\!\,^t\!H'$
\vskip0.1truecm
\item[(iv)]   $(h^\sharp)^\dagger[\Sigma^\perp]=(\Sigma')^\perp$  and  $i_{h^\sharp}(\,\!^t\!H)=\,\!^t(H')$.
\end{itemize}
In this case, conjugation by $g(0,h)$ establishes  a one-to-one correspondence between 
subgroups of $\Sigma\rtimes H$  in $\cE$and subgroups of $\Sigma'\rtimes H'$ in  $\cE$.
\end{prop}
\begin{proof}
The equivalence of (i), (ii), (iii) and (iv) is a matter of writing down the various operations. Clearly, if $\Sigma_0\rtimes H_0$ is a subgroup of $\Sigma\rtimes H$, then $i_{g(0,h)}$ maps it into the subgroup
$(h^\dagger[\Sigma_0])\rtimes(i_h(H_0))$ of $\Sigma'\rtimes H'$, and conversely. 
\end{proof}

\section{Classification of $\cE$ when $d=2$}
As explained in the introduction, our classification problem is
motived by search of all possible signal analyses 
associated with the metaplectic representation of $\Spdr$. In this
framework two groups  $G=\Sigma\rtimes H$ and $G'=\Sigma'\rtimes H'$ 
in the class $\cE$ are  regarded as equivalent if there exists
$w\in\Spdr$ such that $G'=wGw^{-1}$, since they give rise to the same
signal analysis.  However, there are two different cases.  If  $w=g(0,h)\in MA$,
Proposition~\ref{conjugation} shows that the conjugation leaves $\cE$
invariant, so that any element in $MA$ is, in this sense,  ``admissible''. Moreover,
item~(ii) of the same proposition shows that the conjugation
transforms $\Sigma$ and $H$ separately.   On the contrary, 
if $w$ is an element
in $\Spdr$, which is not in $MA$,  then $wGw^{-1}$ is not necessarily in the class
$\cE$.  Lemma~\ref{bruhat} describes the admissible elements $w$,
typically Weyl group elements,  whose conjugation, however, may fail to
preserve the semi-direct product. It can indeed happen that $w g(\sigma,I)w^{-1}\not \in N$ 
for some $g(\sigma,I)\in G$ or that  $wg(0,h)w^{-1}\not \in MA$ for some
$g(0,h)\in G$. However, since $wGw^{-1}$ is a Lie subgroup of $Q$, by Proposition~\ref{triplette}
$wGw^{-1}$  defines a triple $(\Sigma,H,\tau)$ where
$\Sigma$ and $H$ are given by  \eqref{tripledef} and $\tau$ is the
map of the form given by Remark~\ref{tauzero}.  In Section~\ref{full}
this characterization will allow us to check when two groups in different classes
modulo $MA$ are equivalent modulo $w$.

The proof of Theorem~\ref{main} is achieved in two main steps.  We
first classify the groups in the class $\cE$ modulo a conjugation in
$MA$ by the following strategy.

\begin{enumerate}[a)]
\item Since $d=2$, the dimension of $\Sigma$ is either $1$ or $2$.
\item We start from the case  $n=1$ and therefore write $\Sigma=\sp\{\sigma\}$.
By Proposition~\ref{conjugation}, we assume that $\sigma$ is in
Sylvester canonical form (there are only three meaningful
possibilities) and compute in each case $H(\Sigma)$ and its Lie
algebra $\gh(\Sigma)$.  
\item We classify all the Lie subalgebras of  $\gh(\Sigma)$  up to
  conjugation by $H(\Sigma)$ and compute the corresponding connected
  Lie subgroups, thereby obtaining all the subgroups in $\cE$ with
  $n=1$. Note that, since the abelian normal factor
    $\Sigma$ is fixed, the only $MA$-conjugations we can consider are
    those that  leave invariant $\Sigma$, {\em i.e.}, 
    the elements of $H(\Sigma)$.
\item We use Proposition~\ref{perp} and describe all the subgroups in $\cE$ with $n=2$ as those that are dual to some $G$ as before, with $n=1$. Indeed, $\dim{\rm Sym}(2,\R)=3$ and hence
$\dim(\sp\{\sigma\}^\perp)=2$. This completes the picture.
\end{enumerate}
The  lists of groups are given
in Propositions~\ref{list1}, Propositions~\ref{list2} and
Propositions~\ref{list3} according to the three possible choices of
$\Sigma$.

In the second step  we identify the groups in the above three lists
that are conjugated by an element $g$ of $Sp(2,\R)$, which is not in $MA$.
Lemma~\ref{bruhat} characterizes all admissible $g$, see
\eqref{br}, and the possible $\Sigma$ (see Remark~\ref{nonMA}).
Hence, we procede by a case by case analysis, which is  based on
Proposition~\ref{triplette} in the appendix.

\subsection{Reduction to canonical form}\label{sec:CF} Whenever
$\sigma\in\Symdr$,  we write
$$
H(\sigma)=\bigl\{h\in\Gldr:h^\dagger[\sigma]=\lambda\sigma\text{ for some }\lambda\in\R^*\bigr\}
$$
instead of $H(\sp\{\sigma\})$, and also
$$
F(\sigma)=\bigl\{h\in\Gldr:h^\dagger[\sigma]=\pm\sigma\bigr\}.
$$
Both $H(\sigma)$ and $F(\sigma)$ are subgroups of $\Gldr$. We make a first observation.
\begin{prop}\label{isoR}
Assume that $\sigma\neq 0$.
 The map $\varphi:\R_+\times F(\sigma)\to H(\sigma)$
defined by $\varphi(e^t,h)=he^{-t/2}$ is a group isomorphism.
\end{prop}
\begin{proof} First of all, if $(e^t,h)\in\R_+\times F(\sigma)$, then
$$
(he^{-t/2})^\dagger[\sigma]=e^th^\dagger[\sigma]=\pm e^t\sigma
$$
and hence $he^{-t/2}\in H(\sigma)$. Clearly, $\varphi$ is a group homomorphism.  If $he^{-t/2}=I$, then $e^{t/2}I=h\in F(\sigma)$ and it follows that
$e^{t}\sigma=\pm\sigma$. Therefore $t=0$ and $h=I$. Hence $\varphi$ is injective. Finally, take $h\in H(\sigma)$.
Then $h^\dagger[\sigma]=\lambda\sigma$ for some $\lambda\in\R^*$. Upon writing $\lambda={\rm sign}(\lambda)|\lambda|=:\eps e^s$, with $\eps=\pm1$, we get
$$
(e^{s/2}h)^\dagger[\sigma]=e^{-s}h^\dagger[\sigma]=\eps\sigma,
$$
so that $e^{s/2}h\in F(\sigma)$. But then $h=\varphi(e^{s},e^{s/2}h)$, whence surjectivity.
\end{proof}
By Sylvester's law of inertia, there exists $g\in\Gldr$ such that $g^\dagger[I_{pqr}]=\sigma$, where $p+q+r=d$
and $I_{pqr}$ is  the canonical   metric with signature $(p,q,r)$, namely
\[
I_{pqr}=\begin{bmatrix}
   I_p & 0 & 0 \\
   0 & -I_q & 0 \\
   0 & 0 & 0 \end{bmatrix}.
\]
We decompose $F(I_{pqr})=O(p,q,r)\cup O^*(p,q,r)$, where
\begin{align*}
O(p,q,r)&=\{g\in\Gldr:\,^tgI_{pqr}g=I_{pqr}\}     \\
O^*(p,q,r)&=\{g\in\Gldr:\,^tgI_{pqr}g=-I_{pqr}\}
\end{align*}
and observe that $O^*(p,q,r)$ is empty whenever $p\not=q$ because $\,^tgI_{pqr}g$ has signature $(p,q,r)$, whereas $-I_{pqr}$ has signature $(q,p,r)$. The former is a  group, the latter is not, and the product of two elements of $O^*(p,q,r)$ is in $O(p,q,r)$.
\begin{cor}\label{canonicalR}  $H(I_{pqr})=\{e^sh:s\in\R,h\in O(p,q,r)\cup O^*(p,q,r)\}$.
\end{cor}
\begin{proof} Follows from Proposition~\ref{isoR} and the definitions of $O(p,q,r)$ and $O^*(p,q,r)$.
\end{proof}

Summarizing the above reasoning, up to a conjugation be an
  element $g(0,h)\in MA$, we can always assume that $\sigma=I_{pqr}$, and since
  $\sp\{\sigma\}=\sp\{-\sigma\}$, we may choose $p\geq q$. In the case
  $d=2$, there are exactly three interesting possibilities for
  $(p,q,r)$, namely $(2,0,0)$, $(1,1,0)$ and $(1,0,1)$, because the
  case $(0,0,2)$ yields $\sigma=0$. Correspondingly, we define $\sigma_1$, $\sigma_2$ and $\sigma_3$ as in \eqref{5sigma}.

%%%%%%%%%%%%%%%%%%%%%%%%%%%%%%%%%%%%%%%%%%%%%%%%%%%%%%
\subsection{Classification} As outlined earlier, we carry out the classification starting from the canonical forms.  In what follows we often write $\eps$ for a number in $\{\pm1\}$.

%%%%%%%%%%%%%%%%%%%%%%%%%%%%%%%%%%%%%%%%%%%%%%%%%%%%%%
\subsubsection{Signature $(2,0,0)$} 
We recall that the full orthogonal  group $O(2)=O(2,0,0)$ decomposes as
$$
O(2)=SO(2)\cup\Lambda\cdot SO(2),
$$
where $\Lambda$ is  $\sigma_2$, regarded  as a  ``rotation with
negative determinant'', rather than a canonical representative in
$Sp(2,\R)$. 

Corollary~\ref{canonicalR} gives
$$
H(\sigma_1)=\R_+\times O(2).
$$
The  Lie algebra of $H(\sigma_1)$ is
$$
\gh(\sigma_1)=\sotwo\oplus\R=\{\alpha J+ \beta I:\alpha,\beta\in\R\}.
$$
Both $H(\sigma_1)$ and $\gh(\sigma_1)$ are abelian direct sums. The nontrivial Lie subalgebras of $\gh(\sigma_1)$ are its one-dimensional subspaces. We put
$$
\gh_\infty(\sigma_1)=\sp\{J\}
$$
and, for $\alpha\in\R$,
$$
\gh_\alpha(\sigma_1)=\sp\{I+\alpha J\}.
$$
\begin{prop}\label{prop:conj1} Take $\alpha_1,\alpha_2\in\R\cup\{\infty\}$. Then  $\gh_{\alpha_1}(\sigma_1)$ is conjugate to $\gh_{\alpha_2}(\sigma_1)$ by an element of $H(\sigma_1)$ if and only if $\alpha_1=\pm\alpha_2$.
\end{prop}
\begin{proof} Take $g\in H(\sigma_1)$.
Since scalars commute with everything, we  can assume that $g\in O(2)=SO(2)\cup\Lambda\cdot SO(2)$.
Observe that  $R_\theta J R_{-\theta}=J$ and $\Lambda J \Lambda = - J$, so that $\sp\{J\}$ is fixed under conjugation by $g$. It follows that $\gh_\infty(\sigma_1)$ is not conjugate to any other algebra in the class. Finally,
\[
R_\theta (\alpha J+I)R_{-\theta}=(\alpha J+I),\qquad\Lambda \,R_{\theta}\, (\alpha J+I)\, R_{-\theta}\, \Lambda = \Lambda \,(\alpha J+I) \, \Lambda = -\alpha J+I
\]
imply the result.
\end{proof}

Next, we identify the connected Lie subgroups corresponding to the
various Lie algebras and then apply duality, in the sense of
Proposition~\ref{perp}.  Clearly,   the connected Lie subgroups of $\Spdr$,
whose Lie algebra is  $\gh(\sigma_1)$, $\gh_\infty(\sigma_1)$ and
$\gh_\alpha(\sigma_1)$ respectively, are
\begin{align*}
H^0(\sigma_1)&=SO(2)\times\R_+\\
H_\infty(\sigma_1)&=SO(2)     \\
H_\alpha(\sigma_1)&=\{e^tR_{\alpha t}:t\in\R\},\qquad\alpha\in[0,+\infty).
\end{align*}
Here  is the first list of groups in the class $\cE$.
\begin{prop}\label{list1}The following is a complete list, up to
$MA$-conjugation, of the  groups in $\cE$ whose normal factor is equal or orthogonal to $\Sigma_1=\sp\{\sigma_1\}$:
$$
\begin{array}{ll}
\text{ (1.i) } \Sigma_1\rtimes H^0(\sigma_1)\hskip3truecm
&
\text{ (1.iii) } \Sigma_1^\perp\rtimes H^0(\sigma_1)\\
\text{ (1.ii) } \Sigma_1\rtimes H_\alpha(\sigma_1),\text{  with }\alpha\in[0,+\infty]
&
\text{ (1.iv) }\Sigma_1^\perp\rtimes H_\alpha(\sigma_1),\text{  with }\alpha\in[0,+\infty].
\end{array}
$$
\end{prop}
\begin{proof}
Items (1.i) and (1.ii) are clear, and arise by taking first the full two-dimensional algebra $\gh(\sigma_1)$ and then its one-dimensional subalgebras.
Now, the groups $H^0(\sigma_1)$ and $H_\infty(\sigma_1)$ are closed under transposition, whereas $\,^tH_\alpha(\sigma_1)=H_{-\alpha}(\sigma_1)$. However, $\Lambda^\dagger[\Sigma_1^\perp]=\Lambda\Sigma_1^\perp\Lambda=\Sigma_1^\perp$ and
$\Lambda H_{-\alpha}(\sigma_1)\Lambda^{-1}=H_{\alpha}(\sigma_1)$.
Hence, applying
Proposition~\ref{perp} we obtain the groups in (1.iii) and (1.iv). 
\end{proof}
%%%%%%%%%%%%%%%%%%%%%%%%%%%%%%%%%%%%%%%%%%%%%%%%%%%%%%%
\subsubsection{Signature $(1,1,0)$} Here the relevant group is  $O(1,1)=O(1,1,0)$ together with 
$$
O^*(1,1)=\{h\in GL(2,\R):\,^thI_{1,-1}h=- I_{1,-1}\}.
$$
By Corollary~\ref{canonicalR}, we obtain
$$
H(\sigma_2)=\R_+\times\bigl(O(1,1)\cup O^*(1,1)\bigr).
$$
 Its Lie algebra $\gh(\sigma_1)$ can be written as
\begin{equation*}
\gh(\sigma_2)={\mathfrak{so}}(1,1)\oplus\R=\{\alpha \sigma_5+\beta I:\alpha,\beta\in\R\}.
\end{equation*}
The non trivial subalgebras are the vector subspaces of $\gh(\sigma_2)$ of dimension $1$. Put
\begin{equation*}
\gh_\infty(\sigma_2)=\sp\{\sigma_5\}
\end{equation*}
and, for $\alpha\in\R$,
\begin{equation*}
\gh_\alpha(\sigma_2)=\sp\{I+\alpha \sigma_5\}.
\end{equation*}
\begin{prop}\label{prop:conj2} Take $\alpha_1,\alpha_2\in\R\cup\{\infty\}$. Then  $\gh_{\alpha_1}(\sigma_2)$ is conjugate to $\gh_{\alpha_2}(\sigma_2)$ by an element of $H(\sigma_2)$ if and only if $\alpha_1=\pm\alpha_2$.
\end{prop}
\begin{proof}
Take $g\in H(\sigma_2)$.
Since scalars commute with everything, we  can assume that $g\in  O(1,1)\cup O^*(1,1)$. The following relations are straightforward:
\[
O(1,1)=\{\pm A_t, \pm \Lambda A_t:t\in\R\},\quad O(1,1)^*=\sigma_5 \cdot O(1,1).
\]
Since $A_t\,\sigma_5 \,A_t^{-1}=\sigma_5$ and $\Lambda\,\sigma_5\,\Lambda=-\sigma_5$, the algebra $\gh_\infty(\sigma_2)$ is not conjugate to any other one in the class. Finally, we have
\begin{align*}
A_t (I+\alpha \sigma_5)A_t^{-1}&=I+\alpha \sigma_5     \\
\Lambda A_t (I+\alpha \sigma_5) A_t^{-1} \Lambda &= \Lambda (I+\alpha \sigma_5) \Lambda = I-\alpha \sigma_5 \\
\sigma_5(I+\alpha \sigma_5)\sigma_5&=
I+\alpha \sigma_5,
\end{align*}
whence the result.
\end{proof}

Finally, it follows from
$\exp t(I+\alpha \sigma_5)=e^t A_{\alpha t}$
that the connected  subgroups of $Q$ whose Lie algebras  are $\gh(\sigma_2)$, $\gh_\infty(\sigma_2)$ and
$\gh_\alpha(\sigma_2)$, respectively, are
\begin{align*}
H^0(\sigma_2)&=SO^0(1,1)\times\R_+\\
H_\infty(\sigma_2)&=SO^0(1,1)\\
H_\alpha(\sigma_2)&=\{e^t A_{\alpha t}:t\in\R\},\qquad\alpha\in[0,\infty).
\end{align*}
Here  is the second list of groups in the class $\cE$.
\begin{prop}\label{list2}The following is a complete list, up to $MA$-conjugation, of the groups in $\cE$ whose normal factor is equal or orthogonal to $\Sigma_2=\sp\{\sigma_2\}$:
$$
\begin{array}{ll}
\text{ (2.i) } \Sigma_2\rtimes H^0(\sigma_2)
&\text{ (2.iii) }  \Sigma_2^\perp\rtimes H^0(\sigma_2)\\
\text{ (2.ii) } \Sigma_2\rtimes H_\alpha(\sigma_2), \text{with }\alpha\in[0,+\infty]
\hskip0.2truecm
&\text{ (2.iv) } \Sigma_2^\perp\rtimes H_\alpha(\sigma_2),\text{with } \alpha\in[0,+\infty].
\end{array}
$$
\end{prop}
\begin{proof}
Argue as in the proof of Proposition~\ref{list1}, but notice that this time $H^0(\sigma_2)$, 
$H_\infty(\sigma_2)$ and $H_\alpha(\sigma_2)$ are all closed under transposition.
\end{proof}

%%%%%%%%%%%%%%%%%%%%%%%%%%%%%%%%%%%%%%%%%%%%%%%%%%%%%%%
\subsubsection{Signature $(1,0,1)$} The group $O(1,0,1)$ is easily computed to be 
$$
O(1,0,1)=\bigl\{\begin{bmatrix}
\pm1 &0 \\
b&a \\
\end{bmatrix}
:a,b\in\R, a\not=0\bigr\},
$$
and $O^*(1,0,1)=\emptyset$. The Lie algebra of $O(1,0,1)$  is
$$
{\mathfrak{so}}(1,0,1)=\bigl\{\begin{bmatrix}
0 &0 \\
b&a \\
\end{bmatrix}
:a,b\in\R\bigr\}.
$$
Clearly,  the identity component  $O^0(1,0,1)$ is isomorphic to the ``$ax+b$'' group.
By Corollary~\ref{canonicalR}, the symmetrizers are
\begin{align*}
H(\sigma_3)&=\bigl\{\ell_{a,b,c}=\begin{bmatrix}
c&0 \\
b&a \\
\end{bmatrix}
:a,b,c\in\R,\;ac\not=0\bigr\}=T,     \\
\gh(\sigma_3)&=\bigl\{\begin{bmatrix}
c&0 \\
b&a \\
\end{bmatrix}
:a,b,c\in\R\bigr\}
\end{align*}
that is,  the group of all nonsingular lower triangular matrices and its Lie algebra. 
We choose $\{I,\sigma_4,B\}$ as a basis of $\gh(\sigma_3)$,
where $B=[\begin{smallmatrix}
0 & 0 \\
1 & 0 \\
\end{smallmatrix} ]$.

First, we analyze the one-dimensional subalgebras in $\gh(\sigma_3)$ up to conjugation by $H(\sigma_3)$.
To this end, parametrizing as in real projective space $\R{\mathbb P}^2$, we put
\begin{align*}
\gh_\infty(\sigma_3)&=\sp\{I\}     \\
\gh_\gamma(\sigma_3)&=\sp\{\gamma I+B\},\hskip2.5truecm \gamma\in\R     \\
\gh_{\gamma,\beta}(\sigma_3)&=\sp\{\gamma I+\beta B+\sigma_4\},\hskip1truecm \gamma,\beta\in\R.
\end{align*}
\begin{prop}\label{onesigma3} Among the one dimensional Lie algebras listed above, the only conjugacies by elements in $H(\sigma_3)=T$ are the following:
\begin{itemize}
\item[(a)] $\gh_\gamma(\sigma_3)$ is conjugate to $\gh_1(\sigma_3)$,  for every real number $\gamma\not=0$, 
\item[(b)] $\gh_{\gamma,\beta}(\sigma_3)$  is conjugate to  $\gh_{\gamma,\beta'}(\sigma_3)$, for every $\gamma,\beta,\beta'\in\R$.
\end{itemize}
\end{prop}
\begin{proof} A direct computation gives
\begin{subequations}
\begin{align}
&\ell_{a,b,c}\,B\,\ell_{a,b,c}^{-1}=\frac{a}{c}\,B, \label{conj1}\\
&\ell_{a,b,c}\,\sigma_4\,\ell_{a,b,c}^{-1}=-\frac{b}{c}\,B+\sigma_4. \label{conj2}
\end{align}
\end{subequations}
From \eqref{conj1} we infer that $\gh_0(\sigma_3)$ cannot be conjugate to either $\gh_\infty(\sigma_3)$ or to any of the algebras $\gh_\gamma(\sigma_3)$, for any $\gamma\not=0$. Also, \eqref{conj1} yields
$$
\ell_{\gamma,0,1}\,(\gamma I+B)\,\ell_{\gamma,0,1}^{-1}=\gamma(I+B)
$$
and statement (a) follows. Again, \eqref{conj1} yields
$$
\ell_{a,b,c}\,(\gamma I+B)\,\ell_{a,b,c}^{-1}=\gamma I+\frac{a}{c}\,B,
$$
which shows that none of the algebras $\gh_\gamma(\sigma_3)$ can possibly be conjugate to any of the algebras 
$\gh_{\gamma,\beta}(\sigma_3)$. Finally, from \eqref{conj1} and \eqref{conj2} we have
$$
\ell_{a,b,c}\,(\gamma I+\beta B+\sigma_4)\,\ell_{a,b,c}^{-1}=\gamma I+\frac{\beta a-b}{c}\,B+\sigma_4,
$$
whence (b).
\end{proof}

By the above proposition,  the relevant one-dimensional subalgebras of $\gh(\sigma_3)$ are
$\gh_0(\sigma_3)$, $\gh_1(\sigma_3)$, $\gh_\infty(\sigma_3)$ and  the family $\{\gh_{\gamma,0}(\sigma_3):\gamma\in\R\}$. The corresponding one-dimensional connected Lie subgroups of $H(\sigma_3)$ are
\begin{align*}
H_0(\sigma_3)&=
\bigl\{\begin{bmatrix}
1 & 0 \\
t & 1 \\
\end{bmatrix}
:t\in\R\bigr\}\\
H_1(\sigma_3)&=\bigl\{e^t\begin{bmatrix}
                                         1 & 0 \\
                                         t & 1 \\
                                       \end{bmatrix}
:t\in\R\bigr\}\\
H_\infty(\sigma_3)&=\bigl\{e^t\begin{bmatrix}
                                         1 & 0 \\
                                         0 & 1 \\
                                       \end{bmatrix}
:t\in\R\bigr\}\\
H_{\gamma,0}(\sigma_3)&=\bigl\{\begin{bmatrix}
                                         e^{\gamma t} & 0 \\
                                         0 & e^{(\gamma+1)t} \\
                                       \end{bmatrix}
:t\in\R\bigr\},\qquad\gamma\in\R.
\end{align*}
Next we put
\begin{align*}
\gk_{0}(\sigma_3)&=\sp\{I,\sigma_4\}	\\
\gk_\infty(\sigma_3)&=\sp\{I,B\}     \\
\gl_\gamma(\sigma_3)&=\sp\{B,\gamma I+\sigma_4\},\qquad \gamma\in\R.   
\end{align*}

\begin{prop}\label{twosigma3}  Up to conjugation by elements in $H(\sigma_3)$, there are no  two-dimensional Lie subalgebras of $\gh(\sigma_3)$ other than those listed above, which are mutually not conjugate.
\end{prop}
\begin{proof} We begin by observing that  the only non trivial bracket among the elements in $\{I,\sigma_4,B\}$  
is of course $[\sigma_4,B]=B$. Assume that $\gh$ is a two-dimensional subalgebra of $\gh(\sigma_3)$ and suppose that $\gh=\sp\{X_1,X_2\}$, with 
\begin{align*}
X_1&=\alpha_1\sigma_4+\beta_1B+\gamma_1I     \\
X_2&=\alpha_2\sigma_4+\beta_2B+\gamma_2I.
\end{align*}
Evidently, requiring that $\gh$ is a Lie algebra is equivalent to asking that
\begin{equation}
[X_1,X_2]=(\alpha_1\beta_2-\alpha_2\beta_1)B
\label{X1X2}
\end{equation}
belongs to $\gh$. If $B\in\gh$, then this is obvious. In this case we may suppose that $X_1=B$ and consequently $X_2=\alpha \sigma_4+\gamma I$. If $\alpha=0$ we get $\gk_\infty(\sigma_3)$, otherwise we set $\alpha=1$ and we get $\gl_\gamma(\sigma_3)$.
If $B\not\in\gh$, then \eqref{X1X2} yields $(\alpha_1\beta_2-\alpha_2\beta_1)=0$. This means that the vectors 
$\alpha_1\sigma_4+\beta_1B$ and $\alpha_2\sigma_4+\beta_2B$ are linearly dependent; hence there exists a linear combination $\lambda X_1+\mu X_2$ that is equal to $I$, which we choose as basis vector for $\gh$ in place, say, of $X_2$. By subtracting off  $\gamma_1I$ from $X_1$, we may thus suppose that the other basis vector is $X_1=\alpha_1\sigma_4+\beta_1B$. If $\alpha_1=0$, then we get again $\gk_\infty(\sigma_3)$. Hence we put
$\alpha_1=1$ and obtain  that
$$
\gh=\sp\{I,\beta B+\sigma_4\}
$$
for some $\beta\in\R$. But this is conjugate to $\gk_{0}(\sigma_3)$, because by \eqref{conj1} and \eqref{conj2} we have
$$
\ell_{a,b,c}\,(\beta B+\sigma_4)\,\ell_{a,b,c}^{-1}=\frac{\beta a-b}{c}\,B+\sigma_4,
$$
which can be made equal to $\sigma_4$ because $a\not=0$.

It remains to be shown that there are no conjugate pairs in the list.  This follows by inspection, taking into account that the only possibilities are given by \eqref{conj1} and \eqref{conj2}.
\end{proof}

By the above proposition,  the relevant two-dimensional subalgebras of $\gh(\sigma_3)$ are
$\gk_{0}(\sigma_3)$,
$\gk_\infty(\sigma_3)$, and the family
$\{\gl_\gamma(\sigma_3):\gamma\in\R\}$. The corresponding two-dimensional connected Lie subgroups of $H(\sigma_3)$ are
\begin{align*}
K_0(\sigma_3)&=
\bigl\{\begin{bmatrix}
e^t & 0 \\
0 & e^s \\
\end{bmatrix}
:s,t\in\R\bigr\}\\
K_\infty(\sigma_3)&=\bigl\{\begin{bmatrix}
                                         e^t & 0 \\
                                         s & e^t \\
                                       \end{bmatrix}
:s,t\in\R\bigr\}\\
L_{\gamma}(\sigma_3)&=\bigl\{\begin{bmatrix}
                                         e^{\gamma t} & 0 \\
                                         s & e^{(\gamma+1)t} \\
                                       \end{bmatrix}
:s,t\in\R\bigr\},\qquad\gamma\in\R.
\end{align*}

We summarize the above discussion
  in the following proposition, which provides the third list of
  groups in the class $\cE$.
\begin{prop}
\label{list3}
The following is a complete list, up to $MA$-conjugation, of the groups in $\cE$ whose normal factor is equal or orthogonal to $\Sigma_3=\sp\{\sigma_3\}$:
$$
\begin{array}{ll}
\text{ (3.i) } \Sigma_3\rtimes H^0(\sigma_3)
&
\text{ (3.ix) }\Sigma_3^\perp\rtimes \,^tH^0(\sigma_3)\\
\text{ (3.ii) } \Sigma_3\rtimes H_0(\sigma_3)
&
\text{ (3.x) }\Sigma_3^\perp\rtimes \, ^tH_0(\sigma_3)\\
\text{ (3.iii) }\Sigma_3\rtimes H_1(\sigma_3)
&
\text{ (3.xi) }\Sigma_3^\perp\rtimes \,^tH_1(\sigma_3)\\
\text{ (3.iv) }\Sigma_3\rtimes H_\infty(\sigma_3)
&
\text{ (3.xii) }\Sigma_3^\perp\rtimes\,^tH_\infty(\sigma_3)\\
\text{ (3.v) }\Sigma_3\rtimes H_{\gamma,0}(\sigma_3),\;\gamma\in\R
\hskip2truecm
&
\text{ (3.xiii) }\Sigma_3^\perp\rtimes H_{\gamma,0}(\sigma_3),\;\gamma\in\R\\
\text{ (3.vi) }\Sigma_3\rtimes K_0(\sigma_3)
&
\text{ (3.xiv) }\Sigma_3^\perp\rtimes K_0(\sigma_3)\\
\text{ (3.vii) }\Sigma_3\rtimes K_{\infty}(\sigma_3)
&
\text{ (3.xv) }\Sigma_3^\perp\rtimes\,^tK_{\infty}(\sigma_3)\\
\text{ (3.viii) }\Sigma_3\rtimes L_{\gamma}(\sigma_3),\;\gamma\in\R
&
\text{ (3.xvi) }\Sigma_3^\perp\rtimes \,^tL_{\gamma}(\sigma_3),\;\gamma\in\R.
\end{array}
$$
\end{prop}
\vskip0.2truecm
\subsection{Classification modulo $\Spnr$ of $\cE$}\label{full}
The question we want to answer is: when are two groups in $\cE$ conjugate by  
$g\in\Spnr$? We now state the main technical lemma, which is a
consequence of the Bruhat decomposition (see \cite{knapp2002}).  We use the following notation
$$
w_0:=\begin{bmatrix}1&0&0&0\\
0&0&0&-1\\
0&0&1&0\\
0&1&0&0
\end{bmatrix}.
$$

\begin{lemma}\label{bruhat} Suppose that $\Sigma_1\rtimes H_1,\Sigma_2\rtimes H_2\in\cE$ are not conjugate modulo $MA$. If $g\in\Spnr$, is such that $g(\Sigma_1\rtimes H_1)g^{-1}=\Sigma_2\rtimes H_2$, then $g$ is of the form
\begin{equation}
g=g(\sigma',h')^{-1}w_0g(a_0\sigma_4,h),
\label{br}
\end{equation}
for some  $\sigma'\in\Symdr$, $h,h'\in GL(2,\R)$, some $a_0\in\R$.
This can only happen if 
 $h^\dagger[\Sigma_1]\subseteq \sigma_4^\perp$ and $ hH_1h^{-1}\subseteq T$. Furthermore $a_0\neq0$ only if 
 $$
 hH_1h^{-1}\subseteq
 \{\left[\begin{smallmatrix}\alpha&0\\\beta&1\end{smallmatrix}\right]:\alpha>0,\beta\in\R\}.
 $$
\end{lemma}

The proof of Lemma~\ref{bruhat} is based on the Bruhat decomposition of $Sp(2,\R)$, that expresses $Sp(2,\R)$ as the disjoint union
$$
Sp(2,\R)=\bigcup_{w\in W}PwP
$$
of the double cosets $PwP$ of the minimal parabolic group $P$, parametrized by the elements in the Weyl group $W$. More precisely, 
$$
P=\left\{\begin{bmatrix} \ell&0\\ \sigma\ell&\ell^\sharp\end{bmatrix}: \ell\in T, \;\sigma\in\Symdr  \right\},
$$
and, with slight abuse of notation,   a representative\footnote{Formally, $W=N(D)/D$ where $D$ is the maximal torus in $Sp(2,\R)$ consisting of its positive diagonal matrices, and $N(D)$ is its normalizer.
We are indicating a set of representatives in $N(D)$.} of the Weyl group element $w\in W$  may be taken in  $Sp(2,\R)$ as a matrix of the form
$$
w=
\begin{bmatrix}S_+&-S_-\\S_-&S_+\end{bmatrix}
\begin{bmatrix}\pi&0\\0&\pi\end{bmatrix}
$$
where $\pi$ is either $I_2$ or $\sigma_5$, and where $S_-=I_2-S_+$, with $S_+$ one of
$$
s_0=\begin{bmatrix}1&0\\0&0\end{bmatrix}, \quad
s_1=\begin{bmatrix}0&0\\0&1\end{bmatrix}, \quad
I_2=\begin{bmatrix}1&0\\0&1\end{bmatrix}, \quad
0=\begin{bmatrix}0&0\\0&0\end{bmatrix}.
$$
As is well-known, $W$ has $8$ elements. Evidently, $w_0$ corresponds to $S_+=s_0$ and $\pi=I_2$. Notice that, $W$ is a semidirect product and in particular
\begin{equation}
 \begin{bmatrix}\sigma_5&0\\0&\sigma_5\end{bmatrix}
 \begin{bmatrix}s_0&-s_1\\s_1&s_0\end{bmatrix}
  \begin{bmatrix}\sigma_5&0\\0&\sigma_5\end{bmatrix}
  =
 \begin{bmatrix}s_1&-s_0\\s_0&s_1\end{bmatrix}.
\label{nots1}
\end{equation}
Also, notice that $\left[\begin{smallmatrix}\sigma_5&0\\0&\sigma_5\end{smallmatrix}\right]\in MA$.
We remind the reader that in the remaining part of this
  section we use Proposition~\ref{triplette}, which establishes a
  canonical description of the Lie subgroups of $Q$ as triples
and, in particular, the class described in Remark~\ref{tauzero}.

\begin{proof}[Proof of Lemma~\ref{bruhat}] First of all, put  $G_1=\Sigma_1\rtimes H_1$,  
$G_2=\Sigma_2\rtimes H_2$ and, according to the Bruhat decomposition, write $g=p_2^{-1}wp_1$ with $p_1,p_2\in P$ and $w\in W$. Therefore
\begin{equation}
p_2G_2p_2^{-1}=w(p_1G_1p_1^{-1})w^{-1}.
\label{brone}
\end{equation}
Clearly, $F_j:=p_jG_jp_j^{-1}$ is a subgroup of $Q$, for $j=1,2$. Also, we can assume that the permutation factor $\pi$ in $w$ is the identity, because it belongs to $MA\subset Q$. By the same token, by \eqref{nots1}, we can suppose that $S_+\not=s_1$. Our assumption is thus
$$
F_2=wF_1w^{-1}.
$$
The proof now proceeds by inspecting the three remaining cases for $w$.

 Suppose $w=-J$, that is $S_+=0$. Upon writing $F_1=(\Sigma,H,\tau)$ and taking
 any element with $h=I_2\in H$, a straightforward computation gives
 $$
 -J\begin{bmatrix}I_2&0\\\sigma&I_2\end{bmatrix}J
 =\begin{bmatrix}I_2&-\sigma\\0&I_2\end{bmatrix},
 $$
 in contradiction with $F_2\subseteq Q$ unless $\sigma=0$. In this case, though, $G_1\not\in\cE$. Hence we may exclude $w=-J$.
 
 Next, suppose $w=I_4$, that is $S_+=I_2$. Going back to \eqref{brone}, we have then $G_2=pG_1p^{-1}$ for some $p\in P$. But this yields $p\in MA$, against the hypothesis.
 
 Finally, suppose $w=w_0$, namely $S_+=s_0$. The conjugation $gG_1g^{-1}=G_2$ can be formally written as in  \eqref{brone}, with the understanding that under the assumption $w=w_0$ we might have to absorb  into $p_1$ a permutation term $\pi$ coming from the Weyl group. We factor
\[
 p_1
 =\begin{bmatrix}
 1&0&0 & 0 \\
 0&1&0 & 0\\
 c&b&1&0\\
 b&0&0&1
 \end{bmatrix}
 \begin{bmatrix}
 1&0&0 & 0\\
 0&1&0 & 0\\
 0&0&1&0\\
 0&a&0&1
 \end{bmatrix}
 \begin{bmatrix}
 h& 0\\
0 &h^\sharp
 \end{bmatrix}
 =g(b\sigma_5+c\sigma_3,I_2)g(a\sigma_4,h).
 \]
As already observed, we cannot assume that $h\in T$. Now, it is easy to check that
\begin{equation}
w_0g(b\sigma_5+c\sigma_3,I_2)w_0^{-1}
=\begin{bmatrix}
 1&0&0 & 0\\
 -b&1&0 & 0\\
 c&0&1&b\\
 0&0&0&1
 \end{bmatrix}\in Q.
\label{sh}
\end{equation}
Therefore, we
have
\begin{align*}
&p_2G_2p_2^{-1}=w(p_1G_1p_1^{-1})w^{-1} \\
&=[w_0g(b\sigma_5+c\sigma_3,I_2)w_0^{-1}]w_0g(a\sigma_4,h)
G_1g(a\sigma_4,h)^{-1}w_0^{-1} [w_0g(b\sigma_5+c\sigma_3,I_2)w_0^{-1}]^{-1}
\end{align*}
and by \eqref{sh} we can absorb the term in square brackets into $p_2\in Q$.
This proves \eqref{br}, because $p_2=g(\sigma',h')$ and $p_1=g(a_0\sigma_4,h)$ for some $a_0\in\R$.

So far we thus have that \eqref{br} holds with $p_2\in Q$, $w=w_0$ and 
$p_1=g(a_0\sigma_4,h)$. Looking at the right hand side of this version of \eqref{br}, we observe that
$$
p_1G_1p_1^{-1}=(\Sigma', H',\tau_1)=G'
$$
where $h^\dagger[\Sigma_1]=\Sigma'$, $H'=hH_1h^{-1}$ and,
  by Remark~\ref{tauzero},
$$
\tau_1(h')=a_0\left(\sigma_4-h'^\dagger[\sigma_4]\right),
\qquad
h'\in H'.
$$
 We start by writing the elements in $G'$ as
\[
 g_{\tau_1}(\sigma',h')=
 \begin{bmatrix}
I_2& 0\\
\sigma'&I_2
 \end{bmatrix}
  \begin{bmatrix}
 I_2& 0\\
\tau_1(h')&I_2
 \end{bmatrix}
  \begin{bmatrix}
 h'&0\\
0&h'^\sharp
 \end{bmatrix}
 \]
 and then we study the effect of conjugation by $w_0$. We thus parametrize
 $$
 \sigma'= 
 \begin{bmatrix}
c&b\\
b&a
 \end{bmatrix},
 \qquad
 h'=
 \begin{bmatrix}
\alpha&\gamma\\
\beta&\delta
 \end{bmatrix},
 $$
 where we must interpret $a(\sigma'),b(\sigma'),c(\sigma')$ and similarly
 $\alpha(h'), \beta(h'),\gamma(h'), \delta(h')$. 
 Computing, we see that
 $$
 w_0 g_{\tau_1}(\sigma',I_2)w_0^{-1}=
 \begin{bmatrix}
 1&0&0&0\\
 -b&1&0&-a\\
 c&0&1&b\\
 0&0&0&1
 \end{bmatrix}
 $$
 is in $Q$ if and only if $a=a(\sigma')\equiv0$ as a function on $\Sigma'$. This means
  $\Sigma'=h^\dagger[\Sigma_1]\subseteq \sigma_4^\perp$.
 Next
 $$
 w_0g_{\tau_1}(0,h')w_0^{-1}=  \begin{bmatrix}
 *&x\\
*&*
 \end{bmatrix}
 $$
 with
 $$
 x=s_0h's_1-s_1\tau_1(h')h's_1-s_1h'^\sharp s_0.
 $$
 Now, only the first summand has a nonzero entry in the upper-left corner, and it is equal to $\gamma$. Therefore $\gamma=\gamma(h')\equiv0$ s a function on $H'$. This is $ hH_1h^{-1}\subseteq T$. Similarly, only the second summand has a nonzero entry in the lower-left corner.
 Furthermore, since $\tau_1(h')=a_0(\sigma_4-h'^\dagger[\sigma_4])$, we have
\begin{equation*}
\tau_1(h')h'=a_0(\sigma_4h'-h'^\sharp \sigma_4)=
a_0 \begin{bmatrix}
0&\beta/\alpha\delta\\
\beta&\delta-1/\delta
 \end{bmatrix},
\label{tauone}
\end{equation*}
whose lower-right corner is $a_0(\delta-1/\delta)$. Therefore, $a_0$ can be different from zero only if the continuous function $\delta$ is $\delta(h')=\pm1$. However, we are working with connected groups, hence $\delta=1$ and so $H'=hH_1h^{-1}\subseteq\{\left[\begin{smallmatrix}\alpha&0\\\beta&1\end{smallmatrix}\right]
 :\alpha>0,\beta\in\R\}$.
\end{proof} 
 
 \begin{remark}\label{nonMA} The last statement of Lemma~\ref{bruhat} has a consequence for the classification problem.  Fix $G=\Sigma\rtimes H\in\cE$, say for example one among the canonical groups determined in the previous section.  If $G$ is conjugate to some other group via an element {\it not} in $MA$, then there must exist $h\in GL(2,\R)$ such that
  $h^\dagger[\Sigma]\subseteq \sigma_4^\perp$ and $ hHh^{-1}\subseteq T$. The first of these conditions forces the determinant of all elements in $\Sigma$ to be less than or equal to zero.
  Therefore, remembering the notation introduced in \eqref{5sigma}, only the following cases can be considered:
  \begin{enumerate}[(a)]
\item if $n=1$, then either $\Sigma=\Sigma_2$ or $\Sigma=\Sigma_3$;
\item if $n=2$, then  $\Sigma=\Sigma_3^\perp$.
\end{enumerate}
 \end{remark}

 \begin{remark} Lemma~\ref{bruhat} may be formulated in a different way. Given $G\in\cE$, if there exists $g\not\in MA$ such that $gGg^{-1}\in\cE$, then
 there must exist $h_0\in GL(2,\R)$ such that 
\begin{equation}
G=g(0,h_0)^{-1}(\Sigma,H,\tau)g(0,h_0)
\label{conj1bis}
\end{equation}
 with $\Sigma\subseteq \sigma_4^\perp$, $H\subseteq T$ and $\tau(h)=a_0(\sigma_4-h^\dagger[\sigma_4])$, hence $\tau\in\cT$ (see Remark~\ref{tauzero}). In this case there exists $\tau'\in\cT$ such that
\begin{equation}
w_0(\Sigma,H,\tau)w_0^{-1}=(\Sigma',H',\tau').
\label{conj2bis}
\end{equation}
 \end{remark}
 The next lemma shows that if \eqref{conj1bis} and \eqref{conj2bis} hold for some $h_0$, then they also hold for $th_0$, for all $t\in T$. This will be used to put $(\Sigma,H,\tau)$ in canonical form.
 \begin{lemma}\label{bruhat2} Suppose that $\Sigma\subseteq \sigma_4^\perp$, $H\subseteq T$ and $\tau(h)=a_0(\sigma_4-h^\dagger[\sigma_4])$ are such that \eqref{conj2} holds
 with $\tau'\in\cT$, for some symmetric $\tau_0$. Then for all $t\in T$ there exists $\tau''\in\cT$ such that
 $$
 w_0g(0,t)(\Sigma,H,\tau)g(0,t)^{-1}w_0^{-1}=(\Sigma'',H'',\tau'').
 $$
 \end{lemma}
 \begin{proof} We parametrize the lower triangular matrices in $T$ by
\begin{equation}
t= \begin{bmatrix}\alpha&0\\\alpha\beta&\delta\end{bmatrix}.
\label{T}
\end{equation}
Then
 $$
 w_0g(0,t)w_0^{-1}=
 g(\begin{bmatrix}0&\beta\\\beta&0\end{bmatrix},
 \begin{bmatrix}\alpha&0\\0&\delta^{-1}\end{bmatrix}):=g'
 $$
 and so $g'(\Sigma',H',\tau')g'^{-1}=(\Sigma'',H'',\tau'')$, where
\begin{align*}
\Sigma''&=\begin{bmatrix}\alpha&0\\0&\delta^{-1}\end{bmatrix}^\dagger[\Sigma']     \\
H''&=\begin{bmatrix}\alpha&0\\0&\delta^{-1}\end{bmatrix}\,H'\,\begin{bmatrix}\alpha&0\\0&\delta^{-1}\end{bmatrix}^{-1}     \\
\tau''&=\begin{bmatrix}0&\beta\\\beta&0\end{bmatrix}+
 \begin{bmatrix}\alpha&0\\0&\delta^{-1}\end{bmatrix}^\dagger[\tau'] .
\end{align*}
In the last line we have identified $\tau',\tau''\in\cT$ with the corresponding symmetric matrices.
 \end{proof}
 
We apply the above lemmata to our classification problem as follows. Take any group $G$ in canonical form. If there exists $g\not\in MA$ that conjugates $G$ to another group in the class $\cE$, then, by Lemma~\ref{bruhat}, there exists $h\in MA$ that maps the vector part $\Sigma$ inside $\sigma_4^\perp$. By Lemma~\ref{bruhat2} we know that any other $ht\in MA$ can be used for this purpose, with $t\in T$, and hence we can reduce the analysis to three possible cases: $\sigma_4^\perp$ itself if $\Sigma$ is bidimensional, and two cases if $n=1$, as the following proposition clarifies.
\begin{prop}\label{Wcanonical} Suppose that $\Sigma$ is a one dimensional vector subspace of $\sigma_4^\perp$.
Then there exists $t\in T$ such that $t^\dagger[\Sigma]$ is generated by:
\begin{enumerate}[(i)]
\item $\sigma_5$ if the signature is  $(1,1,0)$, and $H(\sigma_5)\cap T$ are the diagonal matrices in $GL(2,\R)$;
\item  $\sigma_3$  if the signature is  $(1,0,1)$, and $H(\sigma_3)=T$.
\end{enumerate}
\end{prop}
\begin{proof}
Denote by
$$
\sigma_0=\begin{bmatrix}c&b\\ b&0\end{bmatrix}
$$
the generator of $\Sigma$ and parametrize as in \eqref{T} the elements in $T$. Then
$$
t^\dagger[\sigma_0]=
\begin{bmatrix}(c\delta-2b\alpha\beta)/\alpha^2\delta&b/\alpha\delta\\
b/\alpha\delta&0
\end{bmatrix}.
$$
If $b=0$ we get case (ii), otherwise we put $b=1$ and we get (i).
\end{proof}
Next we perform the conjugation by $w_0$, as in \eqref{conj2}, under the necessary  conditions on $\Sigma$ and $H$ that must be satisfied, but without assuming that the conjugation produces a group in $\cE$, so that for the resulting triple $(\Sigma',H',\tau')$ we don't know that $\tau'\in\cT$. We start by computing $\Sigma'$ and $H'$.
\begin{lemma}\label{bruhat3} Take $\Sigma\subseteq \sigma_4^\perp$, $H\subseteq T$ and $\tau(h)=a_0(\sigma_4-h^\dagger[\sigma_4])$ with  $w_0(\Sigma,H,\tau)w_0^{-1}\subset Q$
 and parametrize
  $$
 \sigma= 
 \begin{bmatrix}
c(\sigma)&b(\sigma)\\
b(\sigma)&0
 \end{bmatrix}\in\Sigma,
 \qquad
 h=
 \begin{bmatrix}
\alpha(h)&0\\
\beta(h)\alpha(h)&\delta(h)
 \end{bmatrix}\in H.
 $$
 Then  $(\Sigma',H',\tau'):=w_0(\Sigma,H,\tau)w_0^{-1}$ is as follows: $\Sigma'$ consists of all the matrices 
 \begin{equation}
\sigma'=\begin{bmatrix}
c(\sigma)+b(\sigma)\beta(h)&\beta(h)\\
\beta(h)&0
\end{bmatrix}
\label{sigmaprime}
\end{equation}
as  $g_\tau(\sigma,h)$ varies in the subset of $(\Sigma,H,\tau)$ whose elements have  the form
\begin{equation}
\sigma=\begin{bmatrix}c(\sigma)&-a_0\beta(h)\\-a_0\beta(h)&0\end{bmatrix}
\qquad
h=\begin{bmatrix}1&0\\-\beta(h)&1\end{bmatrix},
\label{specialSH}
\end{equation}
the group $H'$ consists of all the matrices 
 \begin{equation}
h'=\begin{bmatrix}
\alpha(h)&0\\
-(a_0\beta(h)+b(\sigma))\alpha(h)&\delta(h)^{-1}
\end{bmatrix}
\label{hprime}
\end{equation}
as  $g_\tau(\sigma,h)$ varies freely  in $(\Sigma,H,\tau)$, and $\tau'$ is not 
necessarily\footnote{See the following Lemma~\ref{bruhat4} for further information on $\tau'$.} in $\cT$.
 \end{lemma}
 \begin{proof}  With our notation, but omitting the various dependencies, we have
 $$
 g_{\tau}(\sigma,h)
 =\begin{bmatrix}
 \alpha&0&0&0\\
 \beta\alpha&\delta&0&0\\
 (c+b\beta)\alpha&b\delta+a_0\beta\delta^{-1}&\alpha^{-1}&-\beta\delta^{-1}\\
( b+a_0\beta)\alpha&a_0(\delta-\delta^{-1})&0&\delta^{-1}
 \end{bmatrix},
 $$
 and hence
 \begin{equation}
w_0g_{\tau}(\sigma,h)w_0^{-1}
 =\begin{bmatrix}
 \alpha&0&0&0\\
 -( b+a_0\beta)\alpha&\delta^{-1}&0&a_0(\delta-\delta^{-1})\\
 (c+b\beta)\alpha&\beta\delta^{-1}&\alpha^{-1}&b\delta+a_0\beta\delta^{-1}\\
 \beta\alpha&0&0&\delta
 \end{bmatrix},
\label{bigmatrix}
\end{equation}
  Rember that the hypothesis  $w_0(\Sigma,H,\tau)w_0^{-1}\subset Q$,
 hence of the form $(\Sigma',H',\tau')$, is equivalent to requiring that 
 $a_0(\delta-\delta^{-1})=0$ (see the proof of Lemma~\ref{bruhat}). 
The upper-left $2\times2$  block is as in \eqref{hprime}, and by setting it to be equal to $I_2$, the lower-left $2\times2$ block is  \eqref{sigmaprime}, and this happens if and only if $g_{\tau}(\sigma,h)$ is as in \eqref{specialSH}. 
 \end{proof}
 
 \begin{remark}\label{nOmega} Observe that  case (i) of Proposition~\ref{Wcanonical} is ruled out from our classification problem by the above lemma. Indeed, in that case, $c(\sigma)=0$ and $h$ is diagonal, so the group elements satisfying \eqref{specialSH} have $\beta(h)=0$, whence $\sigma'=0$.
  \end{remark}
\begin{remark}\label{therest} In both the remaining two cases ($\Sigma=\Sigma_3$ and $\Sigma=\sigma_4^\perp$), we can always take $\beta(h)=0$
and $c(\sigma)=1$ in \eqref{specialSH}. Therefore we always obtain that $\sigma_3\in\Sigma'$. 
As a result, the only canonical groups that are possibly conjugate to other groups in the class $\cE$ are those listed in Proposition~\ref{list3}, because $\sigma_5\Sigma_3^\perp \sigma_5=\sigma_4^\perp$.
 \end{remark}
 
Finally, we complete the picture drawn in Lemma~\ref{bruhat3}.

\begin{lemma}\label{bruhat4}
Hypotheses and notation as in Lemma~\ref{bruhat3}, with either $\Sigma=\Sigma_3$ or $\Sigma=\sigma_4^\perp$. For every pair of real numbers $a',b'$ define the function
\begin{equation*}
\Psi(\sigma,h)=
\beta(h)-b'\left(1-\delta(h)\alpha(h)^{-1}\right)+a'\left(a_0\beta(h)+b(\sigma)\right)\delta(h)^2.
\label{betatau}
\end{equation*}
Then  $\tau'\in\cT$ if and only if there exist $a',b'$ such that   for all $h\in H$ and all  $\sigma\in\Sigma$ 
\begin{equation}
a'(1-\delta^2(h))=0,
\qquad
\begin{bmatrix}1&0\\\Psi(\sigma,h)&1\end{bmatrix}\in H,
\qquad
a_0\Psi(\sigma,h)\sigma_5\in\Sigma.
\label{crazytau}
\end{equation}
In this case, the symmetric matrix associated to $\tau'$ is 
$\begin{bmatrix}0&b'\\b'&a'\end{bmatrix}$.
\end{lemma}
\begin{proof} Look at \eqref{bigmatrix}. The lower-left $2\times2$ block factors as
$$
\begin{bmatrix}c(\sigma)+2b(\sigma)\beta(h)+a_0\beta(h)^2&\beta(h)\\
\beta(h)&0\end{bmatrix}
\begin{bmatrix}\alpha(h)&0\\-(b(\sigma)+a_0\beta(h))\alpha(h)&\delta(h)^{-1}
\end{bmatrix},
$$
where the second is evidently $h'(\sigma,h)$. Now, $\tau'\in\cT$ if and only if the first factor, that we denote by  $\omega(\sigma,h)$, satisfies 
\begin{equation}
\omega(\sigma,h)-\tau'(h'(\sigma,h))
=\omega(\sigma,h)-\bigl(\tau'-h'(\sigma,h)^\dagger[\tau']\bigr)
\in\Sigma'
\label{SigmaPrime}
\end{equation}
for some symmetric (constant) matrix $\tau'$. For any such
$$
\tau'=\begin{bmatrix}c'&b'\\b'&a'\end{bmatrix}
$$
a direct computation gives that $h'(\sigma,h)^\dagger[\tau']$ is equal to
$$
\begin{bmatrix}
c'\alpha^{-2}+2b'\alpha^{-1}\delta\left(a_0\beta+b\right)+a'\delta^2\left(a_0\beta+b\right)^2
&\delta\left[b'\alpha^{-1}+a'\delta\left(a_0\beta+b\right)\right]\\
\delta\left[b'\alpha^{-1}+a'\delta\left(a_0\beta+b\right)\right]&a'\delta^2
\end{bmatrix}
$$
Now, the lower-right entry of $\omega(\sigma,h)-\bigl(\tau'-h'(\sigma,h)^\dagger[\tau']\bigr)$
is $a'(\delta^2-1)$, and must vanish. This is the first   condition in \eqref{crazytau}.
The upper-right entry of $\omega(\sigma,h)-\bigl(\tau'-h'(\sigma,h)^\dagger[\tau']\bigr)$
is precisely $\Psi(\sigma,h)$. As we have already observed, $\sigma_3\in\Sigma'$. Therefore
\eqref{SigmaPrime} holds true if and only if $\Psi(\sigma,h)\sigma_5\in\Sigma'$.
By Lemma~\ref{bruhat3}, this occurs if and only if the remaining two conditions in \eqref{crazytau}
are satisfied.
\end{proof}

\begin{remark}\label{azero} Notice that if there exists $h\in H$ such that $\delta(h)\not=1$ then both $a_0$ and $a'=0$.
\end{remark}

\begin{remark}\label{image} Lemma~\ref{bruhat4} expresses necessary and sufficient conditions for the conjugation via $w_0$ to send a group $(\Sigma,H,\tau)$ with $\tau\in\cT$, in a group of the same kind. The image group, however, is determined in  Lemma~\ref{bruhat3}
\end{remark}

\begin{remark}\label{easycheck} Notice that if we choose $a'=b'=0$, and $a_0=0$, then
\eqref{crazytau} is satisfied if and only if for every $h\in H$
$$
\begin{bmatrix}1&0\\\beta(h)&1\end{bmatrix}\in H.
$$
In this case, $\tau=\tau'=0$ and conjugation by $w_0$ sends the group $\Sigma\rtimes H\in\cE$ into another group in $\cE$, namely $\sigma_4^\perp\rtimes H'$.
\end{remark}

  \begin{proof}[Proof of Theorem~\ref{main}] 
    We are in a position to apply these results to our classification
    problem. We take a group $G\in\cE$ in canonical form and we want
    to know if it is conjugate to another such, or not.  By
    Lemma~\ref{bruhat}, we must find $g\in MA$ (and
    Lemma~\ref{bruhat2} tells us that any such choice is legitimate)
    such that $gGg^{-1}=\Sigma\rtimes H$, where either
    $\Sigma=\Sigma_3$ or $\Sigma=\sigma_4^\perp$, by
    Remark~\ref{nOmega}.  Hence it is enough to consider the groups in
    the list of Proposition~\ref{list3}.  At this point we look at $H$
    and check whether there are entries $\delta(h)\not=1$, in which
    case the condition $a_0(1-\delta(h)^2)=0$ forces $a_0=0$ and the
    first of \eqref{crazytau} forces $a'=0$. If not, we must allow for
    $a_0\not=0$ and $a'\not=0$. Next we verify if the various
    conditions in \eqref{crazytau} are satisfied for some $a',b'$. In
    this case, we know that
    $w_0(\Sigma,H,\tau)w_0^{-1}=(\Sigma',H',\tau')$ with both
    $\tau,\tau'\in\cT$, which is equivalent to saying that
    $\Sigma\rtimes H$ is conjugate to a $\Sigma'\rtimes H'$ yet to be
    determined. Finally, using Lemma~\ref{bruhat3} we find all the
    elements of the form \eqref{specialSH} and thus compute $\Sigma'$
    and $H'$ by means of \eqref{sigmaprime} and \eqref{hprime},
    respectively. The last step is to identify the $MA$-canonical form
    of $\Sigma'\rtimes H'$.

  We now apply the above procedure to the groups of
    Proposition~\ref{list3}.  Each of them is labelled  with the
    corresponding  item  in~Proposition~\ref{list3} and  the
    subsequent heading  enumerates, if they exist, all the
    possible conjugations with other groups listed in
    Proposition~\ref{list3} by writing $G_1\sim G_2$ to mean that
    $G_1$ and $G_2$ are are conjugate. 
With slight abuse of notation, we use
    $\sigma_5$ in place of $g(0,\sigma_5)$.  Recall that $\sigma_5
    \Sigma_3^\perp \sigma_5^{-1} = \sigma_4^\perp$. \vskip0.2truecm
    \paragraph{\bf (3.i):} $\Sigma_3\rtimes T^0\sim
    \Sigma_3^\perp\rtimes K_0(\sigma_3)$, {\em i.e.} item~(3.xiv).
    \begin{enumerate}[(i)]
    \item There is $h\in H$ such that $\delta(h)\neq 1$, hence
      $a_0=a'=0$;
    \item with the choice $\tau'=0$, $\Psi(\sigma,[\begin{smallmatrix}
        \alpha & 0 \\ \beta \alpha & \delta \end{smallmatrix}]
      )=\beta$ with $\beta\in\R$ and $[\begin{smallmatrix} 1 & 0 \\
        \beta & 1 \end{smallmatrix}]\in H$, hence \eqref{crazytau} are
      satisfied;
    \item $\Sigma'=\sigma_4^\perp$, $H'=K_0(\sigma_3)$ and $\sigma_5
      (\sigma_4^\perp\rtimes K_0(\sigma_3) )\sigma_5^{-1}=
      \Sigma_3^\perp\rtimes K_0(\sigma_3)$.
    \end{enumerate}
    \paragraph{\bf (3.ii):} $\Sigma_3\rtimes H_0(\sigma_3)$
   is not conjugate to other groups in the list.
    \begin{enumerate}[(i)]
    \item If $a_0=0$, then $H'=\{I_2\}$ since $\alpha(h)=\beta(h)=1$
      and $b(\sigma)=0$; hence $\Sigma'\rtimes\{I_2\}=\Sigma'$ can
      not be conjugate to an element of the class $\cE$ with $q\in Q$;
    \item if $a_0\not=0$, then $\Sigma'=\Sigma_3$ and
      $H'=H_0(\sigma_3)$.
    \end{enumerate}
    \paragraph{\bf (3.iii)} $\Sigma_3\rtimes H_1(\sigma_3)$  is not conjugate to other groups in the list.
    \begin{enumerate}[(i)]
    \item There is $h\in H$ such that $\delta(h)\neq 1$, hence
      $a_0=a'=0$.
    \item $\Psi(\sigma,[\begin{smallmatrix} e^t & 0 \\ t e^t &
        e^t \end{smallmatrix}] )=t$ with $t\in\R$, but
      $[\begin{smallmatrix} 1 & 0 \\ t & 1 \end{smallmatrix}]\notin H$
      if $t\neq 0$, hence \eqref{crazytau} are not satisfied.
    \end{enumerate}
    \paragraph{\bf (3.iv):}
    $\Sigma_3\rtimes H_\infty(\sigma_3)\sim \Sigma_3 \rtimes
    H_{-\frac12,0}(\sigma_3)$,  {\em i.e.} item~(3.v) with $\gamma=-\frac12$.
    \begin{enumerate}[(i)]
    \item There is $h\in H$ such that $\delta(h)\neq 1$, hence
      $a_0=a'=0$;
    \item with the choice $\tau'=0$, $\Psi(\sigma,[\begin{smallmatrix}
        e^t & 0 \\ 0 & e^t \end{smallmatrix}] )=0$, hence
      \eqref{crazytau} are trivially satisfied;
    \item $\Sigma'=\Sigma_3$ and $H'=\{[\begin{smallmatrix} e^t & 0 \\
        0 & e^{-t} \end{smallmatrix}]:t\in\R \} =
      H_{-\frac12,0}(\sigma_3).$
    \end{enumerate}
\paragraph{\bf (3.v) with $\gamma=-\frac12$:} see (3.iv).
    \paragraph{\bf (3.v) with $\gamma\neq-\frac12$:} $\Sigma_3\rtimes
    H_{\gamma,0}(\sigma_3)\sim \Sigma_3 \rtimes
    H_{-\frac{\gamma}{2\gamma+1},0}(\sigma_3)$.
    \begin{enumerate}[(i)]
    \item There is $h\in H$ such that $\delta(h)\neq 1$, hence
      $a_0=a'=0$;
    \item with the choice $\tau'=0$, $\Psi(\sigma,[\begin{smallmatrix}
        e^{\gamma t} & 0 \\ 0 & e^{(\gamma+1)t} \end{smallmatrix}]
      )=0$, hence \eqref{crazytau} are trivially satisfied;
    \item $\Sigma'=\Sigma_3$ and $H'=\{[\begin{smallmatrix} e^{\gamma
          t} & 0 \\ 0 & e^{-(\gamma+1)t } \end{smallmatrix}]:t\in\R \}
      = H_{-\frac{\gamma}{2\gamma+1},0}(\sigma_3)$.
    \end{enumerate}
    \paragraph{\bf (3.vi):} $\Sigma_3\rtimes K_0(\sigma_3)$  is not conjugate to other groups in the list.
    \begin{enumerate}[(i)]
    \item There is $h\in H$ such that $\delta(h)\neq 1$, hence
      $a_0=a'=0$;
    \item with the choice $\tau'=0$, $\Psi(\sigma,[\begin{smallmatrix}
        e^t & 0 \\ 0 & e^s \end{smallmatrix}] )=0 $, hence
      \eqref{crazytau} are trivially satisfied;
    \item $\Sigma'=\Sigma_3$ and $H'=K_0(\sigma_3)$, so that $w_0
      (\Sigma_3\rtimes K_0(\sigma_3))w_0^{-1}= \Sigma_3\rtimes
      K_0(\sigma_3)$.
    \end{enumerate}
    \paragraph{\bf (3.vii):}
    $\Sigma_3\rtimes K_\infty(\sigma_3) \sim \Sigma_3^\perp\rtimes
    H_{-\frac12,0}(\sigma_3)$, {\em i.e.} item~(3.xiii) with $\gamma=-\frac12$.
    \begin{enumerate}[(i)]
    \item There is $h\in H$ such that $\delta(h)\neq 1$, hence
      $a_0=a'=0$;
    \item with the choice $\tau'=0$, $\Psi(\sigma,[\begin{smallmatrix}
        e^t & 0 \\ s e^t & e^t \end{smallmatrix}] )=s$ with $s \in\R$
      and $[\begin{smallmatrix} 1 & 0 \\ s & 1 \end{smallmatrix}]\in
      H$, hence \eqref{crazytau} are satisfied;
    \item $\Sigma'=\sigma_4^\perp$, $H'=\{[\begin{smallmatrix} e^t & 0
        \\ 0 & e^{-t } \end{smallmatrix}]:t\in\R \}$ and $\sigma_5
      (\sigma_4^\perp\rtimes H')\sigma_5^{-1}= \Sigma_3^\perp\rtimes
      H_{-\frac12,0}(\sigma_3)$.
    \end{enumerate}

    \paragraph{\bf (3.viii) with $\gamma\neq-\frac12$:} $\Sigma_3\rtimes
    L_\gamma(\sigma_3) \sim \Sigma_3^\perp\rtimes
    H_{-\frac{\gamma+1}{2\gamma+1},0}(\sigma_3)$,  {\em i.e.} item~(3.xiii)
      with $\gamma\neq-\frac12$.
\paragraph{\bf (3.viii) with $\gamma=-\frac12$:} $\Sigma_3\rtimes L_{-1/2}(\sigma_3) \sim
    \Sigma_3^\perp\rtimes H_{\infty}(\sigma_3)$,  {\em
        i.e.} item~(3.xii). 

 The proof of the above two cases can be done simultaneously.
    \begin{enumerate}[(i)]
    \item If $\gamma\neq -1$, there is $h\in H$ such that
      $\delta(h)\neq 1$, hence $a_0=a'=0 $; if $\gamma=1$, we choose
      $a_0=0$ (see \ref{gamma1});
    \item with the choice $\tau'=0$, $\Psi(\sigma,[\begin{smallmatrix}
        e^{\gamma t} & 0 \\ s e^{\gamma t} & e^{(\gamma+1)
          t}\end{smallmatrix}] )=s$ with $s \in\R$ and
      $[\begin{smallmatrix} 1 & 0 \\ s & 1 \end{smallmatrix}]\in H$,
      hence \eqref{crazytau} are satisfied;
    \item $\Sigma'=\sigma_4^\perp$, $H'=\{[\begin{smallmatrix}
        e^{\gamma t} & 0 \\ 0 & e^{-(\gamma+1)t
        } \end{smallmatrix}]:t\in\R \}$ and
      $\sigma_5(\sigma_4^\perp\rtimes H' )\sigma_5^{-1} =
      \Sigma_3^\perp\rtimes
      H_{-\frac{\gamma+1}{2\gamma+1},0}(\sigma_3)$ if
      $\gamma\not=-1/2$ and $\sigma_5(\sigma_4^\perp\rtimes H'
      )\sigma_5^{-1} = \Sigma_3^\perp\rtimes H_{\infty}(\sigma_3)$ if
      $\gamma=-1/2$;
    \item\label{gamma1} if $\gamma=-1$ and $a_0\neq 0$, choose $b'=0$
      and $a'=-a_0^{-1}$ so that $\Psi(\sigma,h)=0$, but
      $\Sigma'=\Sigma_3$ and $H'=L_{-1}$ so that $w_0 (\Sigma_3\rtimes
      L_{-1}(\sigma_3))w_0^{-1}= \Sigma_3\rtimes L_{-1}(\sigma_3)$.
    \end{enumerate}
    \paragraph{\bf (3.ix):} $\Sigma_3^\perp \rtimes \,^tH^0(\sigma_3)$
     is not conjugate to other groups in the list.
    \begin{enumerate}[(i)]
    \item $\sigma_5 ^tH^0(\sigma_3) \sigma_5^{-1}= T^0$;
    \item $\Sigma'=\sigma_4^\perp$ and $H'=T^0$.
    \end{enumerate}
    \paragraph{\bf (3.x):} $\Sigma_3^\perp\rtimes \,^tH_0(\sigma_3)$  is not conjugate to other groups in the list.
    \begin{enumerate}[(i)]
    \item $\sigma_5 ^tH_0(\sigma_3) \sigma_5^{-1}= H_0(\sigma_3) $;
    \item $\Sigma'=\sigma_4^\perp$ and $H'=H_0(\sigma_3) $.
    \end{enumerate}
    \paragraph{\bf (3.xi):} $\Sigma_3^\perp\rtimes \, ^t
    H_1(\sigma_3)$  is not conjugate to other groups in the list.
    \begin{enumerate}[(i)]
    \item $\sigma_5 ^tH_1(\sigma_3)\sigma_5^{-1}= H_1(\sigma_3)$;
    \item There is $h\in H$ such that $\delta(h)\neq 1$, hence
      $a_0=a'=0$;
    \item $\Psi(\sigma,[\begin{smallmatrix} e^t & 0 \\ t e^t &
        e^t \end{smallmatrix}] )=t$ with $t\in\R$, but
      $[\begin{smallmatrix} 1 & 0 \\ t & 1 \end{smallmatrix}]\notin H$
      if $t\neq 0$, hence \eqref{crazytau} are not satisfied.
    \end{enumerate}
    \paragraph{\bf (3.xii) with $\gamma=-\frac12$: }   see  (3.viii).
    \paragraph{\bf (3.xiii) with $\gamma\neq-\frac12$:}  see (3.viii)
      with $\gamma\neq-\frac12$.
    \paragraph{\bf (3.xiv):} see (3.i).
    \paragraph{\bf (3.xv) :}    $\Sigma_3^\perp\rtimes\,^tK_{\infty}(\sigma_3)\sim\Sigma_3^\perp\rtimes
    \,^tL_{-\frac12}(\sigma_3)$ {\em i.e.} item~(3.xvi) with $\gamma=-\frac12$.
    \begin{enumerate}[(i)]
    \item $\sigma_5 ^tK_{\infty}(\sigma_3)\sigma_5^{-1}=
      K_{\infty}(\sigma_3)$;
    \item There is $h\in H$ such that $\delta(h)\neq 1$, hence
      $a_0=a'=0$;
    \item with the choice $\tau'=0$, $\Psi(\sigma,[\begin{smallmatrix}
        e^t & 0 \\ s e^t & e^t \end{smallmatrix}] )=s$ with $s \in\R$
      and $[\begin{smallmatrix} 1 & 0 \\ s & 1 \end{smallmatrix}]\in
      H$, hence \eqref{crazytau} are satisfied;
    \item $\Sigma'=\sigma_4^\perp$, $H'=\{[\begin{smallmatrix} e^t & 0
        \\ s e^t & e^{-t } \end{smallmatrix}]:t\in\R \}$ and $\sigma_5
      (\sigma_4^\perp\rtimes H')\sigma_5^{-1}= \Sigma_3^\perp\rtimes
      \,^t L_{-\frac12}(\sigma_3)$.
    \end{enumerate}
\paragraph{\bf (3.xvi) with $\gamma=-\frac12$: see  (3.xv).}
    \paragraph{\bf (3.xvi) with $\gamma\neq-\frac12$:}
    $\Sigma_3^\perp\rtimes \,^tL_{\gamma}(\sigma_3) \sim
    \Sigma_3^\perp\rtimes \,^tL_{-\frac{\gamma}{2\gamma+1}}(\sigma_3)$.
    \begin{enumerate}[(i)]
    \item $\sigma_5 ^t L_{\gamma}(\sigma_3) \sigma_5^{-1}=
      \{[\begin{smallmatrix} e^{(\gamma+1) t} & 0 \\ s e^{(\gamma+1)
          t}& e^{\gamma t} \end{smallmatrix}]
      :s,t\in\R\}=L_{-(\gamma+1)}(\sigma_3)$;
      % \item If $\gamma\neq 0$, there is $h\in H$ such that
      %   $\delta(h)\neq 1$, hence $a_0=a'=0 $; if $\gamma=1$, we
      %   choose $a_0=0$ (see
      %   \ref{gamma0});
    \item with the choice $\tau'=0$, $\Psi(\sigma,[\begin{smallmatrix}
        e^{(\gamma+1) t} & 0 \\ s e^{(\gamma+1) t}& e^{\gamma
          t} \end{smallmatrix}] )=s$ with $s \in\R$ and
      $[\begin{smallmatrix} 1 & 0 \\ s & 1 \end{smallmatrix}]\in H$,
      hence \eqref{crazytau} are satisfied;
    \item $\Sigma'=\sigma_4^\perp$, $H'=\{[\begin{smallmatrix}
        e^{(\gamma+1) t} & 0 \\ s e^{(\gamma+1) t} & e^{-\gamma t
        } \end{smallmatrix}]:t,s\in\R \}$ and, since
      $\gamma\not=-1/2$,
$$ 
\sigma_5 H' \sigma_5^{-1} = \{[\begin{smallmatrix} e^{-\gamma t } & s
  e^{(\gamma+1) t} \\ 0 & e^{(\gamma+1) t} \end{smallmatrix}]:t,s\in\R
\} = \,^t\!L_{-\frac{\gamma}{2\gamma+1}}(\sigma_3).
$$
\end{enumerate}
\end{proof}

%%%%%%%%
\appendix
\section{The parabolic group $Q$ and its subgroups}\label{sec:parabolic}
 In this section we characterize the Lie subgroups of $Q$.  We need
 the ``if part''  to study the conjugation in $\Spdr$. However we
 think that the result is of some independent interest and we hope that it might help
 to classify all the reproducing groups of  $Q$. 

Here, we denote by $\pi:Q\to\Gldr$ 
the  smooth group homomorphism $g(\sigma,h)\mapsto h$. 
\begin{prop}\label{triplette} 
Take $G$ be a Lie subgroup of $Q$ and define 
\begin{equation}
  \label{tripledef}
  H=\pi(G) \qquad \Sigma=\{\sigma\in\Symdr: g(\sigma,I)\in G\}.
\end{equation}
Then $H$ is a Lie subgroup of $\Gldr$, $\Sigma$ is a Lie subgroup of $\Symdr$ which is invariant under the action  $\sigma\mapsto h^\dag[\sigma]$ and  there exists a measurable map $\tau:H\to\Symdr$ that satisfies
\begin{align}
&\tau(I)=0\nonumber\\
&\tau(h)+h^{\dagger}[\tau(h')]-\tau(hh')\in\Sigma\label{cocycle}
\end{align}
for every $h,h'\in H$. The triple $(\Sigma, H,\tau)$ identifies the group, in the sense that
 \begin{equation}
   \label{Gdef}
   G = \{g(\sigma+\tau(h),h):\sigma\in\Sigma,\, h\in H\}.
 \end{equation}
Conversely,  if $(\Sigma,H,\tau)$ is any such triple, then   $G$ as in~\eqref{Gdef}  is a Lie subgroup of $Q$ satisfying~\eqref{tripledef}. 
\end{prop}

\begin{proof} The first part is essentially known, see  \cite{dk00}, Proposition~1.11.8. We sketch the main steps. Take a Lie subgroup $G$ of $Q$. A standard result on Lie groups, see e.~g.
Theorem 2.7.3 in \cite{raja84}, ensures that $H:=\pi(G)$ is a Lie subgroup  of $\Gldr$.  
Since $\ker(\pi)$ is closed  in $G$, hence a Lie subgroup of $Q$, the set
$$
\Sigma=\bigl\{\sigma\in\Symdr:g(\sigma,e)\in G\bigr\}\simeq\ker(\pi)
$$
is a Lie subgroup of $\Symdr$,  and is contained in $H(\Sigma)$ (recall \eqref{largest}) because $\ker(\pi)$ is normal in $G$. 
The quotient Lie group $H=G/\ker(\pi)$ admits a global measurable
section $s:H\to G$ %smooth in a neighborhood of the identity  and
that maps $I$ to $g(0,I)$ (see \cite{mackey52} or \cite{raja85}). Since $G\subset Q$, we
may write 
$s(h)=g(\tau(h),h)$. Therefore, if $g\in G$, then we may write $g=g(\sigma+\tau(h),h)$, where $h=\pi(g)$ and $\sigma\in\Sigma$. Since $G$ is a group, the product  \eqref{semidirect} shows that 
\begin{equation}
\sigma+h^\dagger[\sigma']+\left(\tau(h)+h^{\dagger}[\tau(h')]-\tau(hh')\right)\in\Sigma
\label{sympiece}
\end{equation}
so that $\tau(h)+h^{\dagger}[\tau(h')]-\tau(hh')\in\Sigma$.

Conversely, fix a triple $(\Sigma,H,\tau)$ as in the statement. We
prove that there exists a Lie subgroup $G$ of $Q$ such that \eqref{tripledef} holds. Define
$G$ as in \eqref{Gdef},  a subgroup of $\Spdr$ because
$$
g(\sigma+\tau(h),h)g(\sigma'+\tau(h'),h')
=g(\sigma+h^\dagger[\sigma']+\left(\tau(h)+h^{\dagger}[\tau(h')]-\tau(hh')\right)+\tau(hh'),hh')
$$
and by the assumptions
$$
\sigma+h^\dagger[\sigma']+\left(\tau(h)+h^{\dagger}[\tau(h')]-\tau(hh')\right)\in\Sigma.
$$
A similar argument applies to inverses.
In order to prove that $G$ is a Lie subgroup, we follow this strategy: first we show that $G$ is a
standard Borel group\footnote{For notation and basic results on these issues, see \cite{raja85}, Chapter~VIII.} with an invariant $\sigma$-finite measure. As a
consequence of a theorem of Mackey's, we will be able to endow $G$ with the  Weil topology, so that 
$G$ becomes  a locally compact second countable
group. Finally, applying a classical result on Lie groups we see that $G$
admits a unique smooth structure converting it into a Lie subgroup of
$\Spdr$.

We claim that $G$ is  a Borel subset of  $\Spdr$.  Since $\Sigma$ and
$H$  are Lie groups,  they are standard Borel spaces with respect to
the corresponding Borel $\sigma$-algebras $\cB(\Sigma)$ and $\cB(H)$. Hence
the  product $\Sigma\times H$ is a standard Borel space with respect
to $\cB(\Sigma)\otimes\cB(H)$ and the injection $\xi:\Sigma\times
H\to\Spdr$, $\xi(\sigma,h)= g(\sigma+\tau(h),h)$, is a Borel
measurable map.  Since $\xi$ is a one-to-one map from a standard Borel
space into another standard Borel space, its range $G$ is a Borel
subset of $\Spdr$ and $\xi$ is a Borel isomorphism  from $\Sigma\times
H$ onto $G$, the latter being endowed with the restriction of $\cB(\Spdr)$. 

We choose  (left) Haar measures $d\sigma$ and $dh$ on $\Sigma$ and
$H$, respectively.  For any fixed $h\in H$, the map $\sigma\mapsto h^\dag[\sigma]$ is a group
homomorphism of $\Sigma$ onto itself, so that the image measure of
$d\sigma$ under $h^\dag[\cdot]$ is again a Haar measure. Hence there
exists a unique $\alpha(h)>0$ such that for all positive Borel measurable
functions $\varphi$ on $\Sigma$
$$
\int_{\Sigma} \varphi(h^\dag[\sigma])d\sigma= \alpha(h)  \int_{\Sigma}\varphi(\sigma)d\sigma.
$$
Since $h\mapsto \int_{\Sigma} \varphi(h^\dag[\sigma])d\sigma$ is Borel
measurable, so is $h\mapsto\alpha(h)$. Furthermore, the uniqueness of
$\alpha(h)$ implies that $h\mapsto \alpha(h)$ is a group homomorphism of
$H$, that is,  $\alpha$ is a continuous positive character of $H$.  
Write $dg$ as the image measure of the  measure $\alpha\cdot d\sigma\otimes dh$
under $\xi$. 
We claim that $dg$  is a $G$-invariant  $\sigma$-finite measure on $G$. Since
both $\Sigma$ and $H$ are $\sigma$-compact and $\alpha$ is continuous,
then $\alpha\cdot d\sigma\otimes dh$ is $\sigma$-finite as well
as $dg$. Moreover, for any  positive Borel measurable
function $\varphi$ on $G$ and $g_0=g(\sigma_0+\tau(h_0),h_0)\in G$
\begin{align*}
  \int_{G} \varphi(g_0g) dg& = \int_{\Sigma\times H}
  \varphi(g(\sigma_0+h_0^\dag[\sigma]+\tau(h_0)+h_0^\dag[\tau(h)], h_0h)) \alpha(h)
d\sigma\,dh \\
& = \int_{H}\int_{\Sigma}
  \varphi(g({\sigma_0+\sigma'+\tau(h_0)+h_0^\dag[\tau(h)]}, h_0h))
\alpha(h_0 h)d\sigma'\,dh  \\
& =\int_H \int_{\Sigma}
  \varphi(g(\sigma'' +\tau(h_0h), h_0h)) \alpha(h_0 h)d\sigma''\,dh\\
& =\int_{\Sigma}\int_H \varphi(g(\sigma'' +\tau(h'), h'))
\alpha( h)d\sigma''\,dh'=\int_{G} \varphi(g) dg,
\end{align*}
where the equality in the second line is due to Fubini's theorem, the
change of variable $h_0^\dag[\sigma]=\sigma'$ and the fact that
$\alpha$ is a character; the equality in the third line
is a consequence of the fact that  $d\sigma$ is the Haar measure on
$\Sigma$; finally, the fourth line follows by Fubini's theorem, the change of
variable $h'=h_0h$ and the $H$-invariance of $dh$. 

Next we apply the theorem of Mackey's, see for example Theorem 8.41 of \cite{raja85}, that
states that there exists exactly one topology on $G$ which converts it into a
locally compact second countable space whose Borel structure is
the original one.  From now on, we regard $G$ as endowed with this  topology. 

Clearly,
the inclusion $i$ of $G$ into $\Spdr$ is a  Borel measurable group
homomorphism. Hence $i$ is continuous, (see Lemma 8.28 of
\cite{raja85}). Finally, by Proposition~1 Ch. IV \S. XIV in \cite{che46}, there
exists exactly one $C^\infty$-structure on $G$ which converts it into a
Lie group and Proposition~1 Ch. IV \S.~XII in \cite{che46} implies that the inclusion
is a $C^\infty$-map. Hence, $G$ is a Lie subgroup of $\Spdr$.
\end{proof}

\begin{remark}\label{equivtau}
The correspondence between the triples
$(\Sigma,H,\tau)$ and the Lie subgroups $G$ of $Q$ is not one-to-one. Indeed,
two different maps $\tau$ and $\tau'$ define the same group $G$ if and
only if $\tau'(h)-\tau(h)\in\Sigma$ for all $h\in H$ and, if
this happens, we say
that $\tau$ and $\tau'$ are equivalent. From now on we thus parametrize the Lie subgroups of $Q$ writing $G=(\Sigma,H,\tau)$, with the understanding that $\tau$ is only defined up to equivalence.
\end{remark}

\begin{remark} One could go about the proof of Proposition~\ref{triplette} in a different way, using the standard result according to which, under the foregoing assumptions, there exists a locally defined smooth section, hence one could assume $\tau$ to be smooth around the identity, and then use this to define a smooth atlas on $G$ via translations.
\end{remark}

\begin{remark}
 The problem of characterizing the Lie subgroups of $Q$ can be stated
 in a slightly different form, in the framework of Lie group extensions \cite{hoc51}.  Since $Q$ is the semi-direct product of
 $\Symdr$ and $\Gldr$, $Q$ is a  (Lie group) extension of $\Symdr$ by
 $\Gldr$. In the language of group extensions,  $i_0:\Symdr\to Q$ (the canonical injection) and
 $\pi_0:Q\to \Gldr$ (the  canonical surjection) give rise to a short exact sequence, that is $i_0(\Symdr)=\ker\pi_0$.\\
Proposition~\ref{triplette} shows that any Lie subgroup $G$ of $Q$ is a  Lie group
 extension of  $\Sigma=\ker\pi$ (a Lie subgroup of $\Symdr$) by $H=\pi(G)$ (a Lie subgroup of $\Gldr$). Furthermore, the canonical inclusion $j$ is a
 group homomorphism of $G$ into $Q$ compatible with $i_0$
 and $\pi_0$, in the sense that the diagram
 $$
 \begin{CD} 
 \Sigma @>i>> G @> \pi>>H\\
 @VV V @VVjV@VV V\\
\Symdr@>i_0>> Q @>\pi_0>>  \Gldr \end{CD}
 $$
 \\
 commutes, where the vertical arrows are the natural inclusions. The factor sets corresponding to the extension $G$ are:
the map  
$$
(h,h')\mapsto \tau(h)+h^{\dagger}[\tau(h')]-\tau(hh')
$$
from $H\times H$ into $\Sigma$ and the map $h\mapsto h^\dag[\cdot]$ from $H$ into the group automorphisms of $\Sigma$.
Conversely,  for any pair $(G,j)$ where $G$ is a Lie group extension of
a  Lie subgroup of $\Symdr$ by  a Lie subgroup of $\Gldr$  and where $j:G\to Q$
is a group homomorphism compatible with both $i_0$ and
$\pi_0$,  $j(G)$ turns out to be a Lie subgroup of  $Q$. 

For any fixed $\Sigma$ and $H$, the  maps $\tau$
satisfying~\eqref{cocycle} characterize all the extensions $G$ of $\Sigma$
by $H$ for which there is a  group homomorphism compatibile with $i_0$ and
$\pi_0$.
\end{remark}

\begin{remark}\label{zero} Several special instances of
  \eqref{cocycle} are of interest. The easiest  is when $\tau$ is
  (equivalent to) zero, a case that plays a prominent r\^ole in our
  paper. When this happens, $G$ becomes the semi-direct product
  $\Sigma\rtimes H$, because \eqref{sympiece} reduces to
  $\sigma+h^\dagger[\sigma']$. The family of subgroups of $Q$ for
  which  $\tau=0$ and both factors are connected and nontrivial, is
  precisely the class $\cE$. 
\end{remark}
The following class of triples is one of the
  technical tools on which  the study of conjugation by an element in
  $\Spdr$, which is not in $MA$, is based. In particular, we introduce the class
$\mathcal T$ of maps $\tau$ we use in Section~\ref{full}.
\begin{remark}\label{tauzero} The next simpler case is when $\tau(h)=\tau_0-h^\dagger[\tau_0]$ for some  $\tau_0\in\Symdr$. The class of maps $\tau$ of this kind will be denoted by $\cT$. This happens  if and only if  we conjugate a group $\Sigma\rtimes H$  by means of $g(\tau_0,I)$:
$$
g(\tau_0,I)g(\sigma,h)g(\tau_0,I)^{-1}=g(\sigma+\tau_0-h^\dagger[\tau_0],h).
$$
We shall often identify the functions $\tau\in\cT$ with the symmetric matrices that uniquely determine them.
For example, we can take
$$
\Sigma=\bigl\{\begin{bmatrix}x&0\\0&0\end{bmatrix}:x\in\R\bigr\},
\qquad
H=\Bigl\{\begin{bmatrix}1&0\\y&1\end{bmatrix}:y\in\R\Bigr\},
$$
thereby obtaining $\Sigma\rtimes H$, consisting of the symplectic matrices
$$
\begin{bmatrix}1&0&0&0\\
y&1&0&0\\
x&0&1&-y\\
0&0&0&1
\end{bmatrix}.
$$
If we conjugate $\Sigma\rtimes H$ with $g(\tau_0,I)$, where $\tau_0$ is the symmetric matrix
$$
\tau_0=\begin{bmatrix}1&0\\0&-1\end{bmatrix},
$$
we obtain 
$$
G=\Bigl\{
\begin{bmatrix}1&0&0&0\\
y&1&0&0\\
x+y^2&y&1&-y\\
2y&0&0&1
\end{bmatrix}:
x,y\in\R\Bigr\},
$$
a subgroup of $Q$ of the form $(\Sigma,H,\tau)$ for which $\tau$ is not equivalent to zero. 
\end{remark}
 For  completeness, we add two examples. The first shows 
that there exist connected  Lie subgroups of $Q$, which are not in the class
$\cE$. The second exhibits a normal abelian factor that is not connected, {\em i. e.} a case where
$\Sigma$ is not a subspace. We stress that this pathology does not
occur for  the triples with $\tau\in\mathcal T$ and, in particular,
for the groups in the class $\cE$ where $\tau=0$.
\begin{remark}\label{giovanni} A slightly more general  class of groups  $G=(\Sigma,H,\tau)$, that includes the previous one,  corresponds to maps $\tau$ that  satisfy 
$$
\tau(h)+h^{\dagger}[\tau(h')]-\tau(hh')=0.
$$
Then $H_{\tau}=\{g(\tau(h),h)\in G:
h\in H\}$ is a Lie subgroup of $Q$. Using the same arguments as those in the proof of Proposition~\ref{triplette}, one sees that $\tau$ is a
$C^\infty$ map from $H$ into $\Symdr$. Furthermore, $G$ is the semi-direct product
of $\Sigma$ and $H_{\tau}$, and is isomorphic (as Lie group) to
$\Sigma\rtimes H$ via the mapping $\Sigma\rtimes H\to G$ given by $(\sigma,h)\mapsto(\sigma+\tau(h),h)$. 
For example, take  
$$
G=\Bigl\{
\begin{bmatrix}e^t&0&0&0\\
0&e^{-t}&0&0\\
se^t&-te^{-t}&e^{-t}&0\\
-te^t&0&0&e^t
\end{bmatrix}:
t,s\in\R\Bigr\}.
$$
Here 
$$
h=\begin{bmatrix}e^t&0\\0&e^{-t}\end{bmatrix},
\qquad
\sigma=\begin{bmatrix}s&0\\0&0\end{bmatrix},
\qquad
\tau(h)=\begin{bmatrix}0&-t\\-t&0\end{bmatrix}.
$$
It is easily checked that $\tau$ is not of the form $\tau(h)=\tau_0-h^\dagger[\tau_0]$ for any symmetric $\tau_0$,  
but
$\tau(h)+h^{\dagger}[\tau(h')]-\tau(hh')=0$.

\end{remark}

\begin{remark}\label{connected}   If $G=(\Sigma,H,\tau)$ is connected, then so is $H$, but $\Sigma$ may well be disconnected. The statement concerning $H$ is clear, since $\pi$ is continuous. Consider 
$$
G=\{\begin{bmatrix}R_\theta&0\\\theta R_\theta&R_\theta\end{bmatrix}:\theta\in\R\}\subset Q.
$$
Clearly $G$ is connected, but \eqref{tripledef} tells us that $\Sigma=2\pi\mathbb Z$, and is therefore not connected.   

Finally, take $\tau:H\to {\rm Sym}(2,\R)$ such that $G=(\Sigma,H,\tau)$.  We show that it is not possible to choose $\tau$ is such a way that
$\tau(h)+h^{\dagger}[\tau(h')]-\tau(hh')=0$ for all $h,h'\in H$. Assuming the converse, 
then $h\mapsto g(\tau(h),h)$ is an injective  measurable (hence smooth)  group homomorphism of the compact group $H$ into $G$. However, $G$ is isomorphic to $\R$, so that it  does not have compact subgroups other than $\{0\}$. 
\end{remark}

%%%%%%%%%%%%%%%%

%\bibliographystyle{amsplain}
%\bibliography{biblio}

\begin{thebibliography}{10}


\bibitem{che46}
C.~Chevalley, \emph{Theory of {L}ie {G}roups. {I}}, Princeton Mathematical
  Series, vol. 8, Princeton University Press, Princeton, N. J., 1946.

\bibitem{codenota06b}
E.~Cordero, F.~De~Mari, K.~Nowak, and A.~Tabacco, \emph{Reproducing groups for
  the metaplectic representation}, Pseudo-differential operators and related
  topics, Oper. Theory Adv. Appl., vol. 164, Birkh\"auser, Basel, 2006,
  pp.~227--244.

\bibitem{codenota06a}
E.~Cordero, F.~De~Mari, K.~Nowak, and A.~Tabacco,
  \emph{Analytic features of reproducing groups for the metaplectic
  representation}, J. Fourier Anal. Appl. \textbf{12} (2006), no.~2, 157--180.

\bibitem{czki12}
J.~Czaja, and E. King, \emph{ Isotropic Shearlet Analogs for
  $L^2(\R^k)$ and Localization Operators},
Numer. Funct. Anal. Optim. \textbf{33} (2012), no.~7-9, 872--905.

\bibitem{dakustte09}
S.~Dahlke, G.~Kutyniok, G.~Steidl, and G.~Teschke,
  \emph{Shearlet coorbit spaces and associated {B}anach frames}, Appl. Comput.
  Harmon. Anal. \textbf{27} (2009), no.~2, 195--214. 

\bibitem{dastte10}
S.~Dahlke,  G.~Steidl, and G.~Teschke, \emph{The continuous
  shearlet transform in arbitrary space dimensions}, J. Fourier Anal. Appl.
  \textbf{16} (2010), no.~3, 340--364. 

\bibitem{Dau92}
I.~Daubechies, \emph{Ten lectures on wavelets}, CBMS-NSF Regional
  Conference Series in Applied Mathematics, vol.~61, Society for Industrial and
  Applied Mathematics (SIAM), Philadelphia, PA, 1992. 

\bibitem{DeDe11} 
F.~De~Mari and E.~De~Vito, \emph{Admissible vectors for mock metaplectic representations},
Appl. Comput. Harmon. Anal. (2012) DOI 10.1016/j.acha.2012.04.001 (in press).


\bibitem{dk00}
J.~J. Duistermaat and J.~A.~C. Kolk, \emph{Lie groups}, Universitext,
  Springer-Verlag, Berlin, 2000.

\bibitem{fol89}
G.~Folland, \emph{Harmonic analysis in phase space}, Annals of
  Mathematics Studies, vol. 122, Princeton University Press, Princeton, NJ,
  1989. 

\bibitem{fuhr05}
H.~F{\"u}hr, \emph{Abstract harmonic analysis of continuous wavelet
  transforms}, Lecture Notes in Mathematics, vol. 1863, Springer-Verlag,
  Berlin, 2005.

\bibitem{ober10}
D.~Labate,  G.~Kutyniok (ed.), \emph{Mini-workshop: Shearlets}, no. Report No.
  44/2010, Mathematisches Forschungsinstitut Oberwolfach, 2010.

\bibitem{gro01}
K.~Gr{\"o}chenig, \emph{Foundations of time-frequency analysis}, Applied
  and Numerical Harmonic Analysis, Birkh\"auser Boston Inc., Boston, MA, 2001.

\bibitem{gukula06}
K.~Guo, G.~Kutyniok, and D.~Labate, \emph{Sparse multidimensional
  representations using anisotropic dilation and shear operators}, Wavelets and
  splines: {A}thens 2005, Mod. Methods Math., Nashboro Press, Brentwood, TN,
  2006, pp.~189--201. 

\bibitem{hoc51}
G.~Hochschild, \emph{Group extensions of {L}ie groups}, Ann. of Math. (2)
  \textbf{54} (1951), 96--109. 

\bibitem{knapp2002}
A.~Knapp, \emph{Lie groups beyond an introduction}, second ed.,
  Progress in Mathematics, vol. 140, Birkh\"auser Boston Inc., Boston, MA,
  2002. 

\bibitem{kula09}
G.~Kutyniok and D.~Labate, \emph{Resolution of the wavefront set using
  continuous shearlets}, Trans. Amer. Math. Soc. \textbf{361} (2009), no.~5,
  2719--2754. 

\bibitem{LWWW}
R.~S. Laugesen, N.~Weaver, G.~L. Weiss, and E.~N. Wilson, \emph{A
  characterization of the higher dimensional groups associated with continuous
  wavelets}, J. Geom. Anal. \textbf{12} (2002), no.~1, 89--102. 
  
\bibitem{mackey52}
G.~Mackey, \emph{Induced representations of locally compact groups.
  {I}}, Ann. of Math. (2) \textbf{55} (1952), 101--139.

\bibitem{mallat09}
S.~Mallat, \emph{A wavelet tour of signal processing}, third ed.,
  Elsevier/Academic Press, Amsterdam, 2009, The sparse way, With contributions
  from Gabriel Peyr{\'e}.

\bibitem{CoMe97}
Y.~Meyer and R.~Coifman, \emph{Wavelets}, Cambridge Studies in Advanced
  Mathematics, vol.~48, Cambridge University Press, Cambridge, 1997,
  Calder{\'o}n-Zygmund and multilinear operators, Translated from the 1990 and
  1991 French originals by David Salinger.

\bibitem{raja84}
V.~S. Varadarajan, \emph{Lie groups, {L}ie algebras, and their
  representations}, Graduate Texts in Mathematics, vol. 102, Springer-Verlag,
  New York, 1984, Reprint of the 1974 edition.

\bibitem{raja85}
\bysame, \emph{Geometry of quantum theory}, second ed., Springer-Verlag, New
  York, 1985.

\end{thebibliography}
\providecommand{\bysame}{\leavevmode\hbox to3em{\hrulefill}\thinspace}
\providecommand{\MR}{\relax\ifhmode\unskip\space\fi MR }
\providecommand{\MRhref}[2]{%
  \href{http://www.ams.org/mathscinet-getitem?mr=#1}{#2}
}
\providecommand{\href}[2]{#2}

\end{document}